\documentclass[a4paper,10pt]{article}
\usepackage[utf8]{inputenc}

\usepackage{wrapfig}
\usepackage{amsmath} 
\usepackage{amsthm} 
\usepackage{amssymb} 
\usepackage{enumerate} 
\usepackage{esint} 
\usepackage{pgf,tikz} 
\usetikzlibrary{arrows} 
 \usepackage{yfonts} 
 \usepackage{mathrsfs} 
 \usepackage{mathabx} 
 \usepackage{graphicx}
\usepackage{caption}
\usepackage{subcaption}
\usepackage{mathtools} 

\definecolor{zzttqq}{rgb}{0.6,0.2,0.}
\definecolor{zzuuuu}{rgb}{0,0,1}
\definecolor{uququq}{rgb}{0.25098039215686274,0.25098039215686274,0.25098039215686274}
\definecolor{cqcqcq}{rgb}{0.7529411764705882,0.7529411764705882,0.7529411764705882}

\newcommand*\circled[1]{\tikz[baseline=(char.base)]{
            \node[shape=circle,draw,inner sep=2pt] (char) {#1};}}

\newcommand*\squared[1]{\tikz[baseline=(char.base)]{
            \node[shape=rectangle,draw,inner sep=2.4pt] (char) {#1}; \node[shape=rectangle,draw,inner sep=1pt] (char) {#1};}}

\newcommand{\C}{{\mathbb C}}       
\newcommand{\R}{{\mathbb R}}       
\newcommand{\N}{{\mathbb N}}       
\newcommand{\Z}{{\mathbb Z}}       
\newcommand{\DDD}{{\mathbb D}}
\newcommand{\Sh}{{\mathbf {Sh}}} 
\newcommand{\SH}{{\mathbf {SH}}} 
\newcommand{\diam}{{\rm diam}}
\newcommand{\dist}{{\rm dist}}
\newcommand{\Dist}{{\rm D}}

\newcommand{\rf}[1]{{(\ref{#1})}}
\newcommand{\supp}{{\rm supp}}

\newcommand{\norm}[1]{{\left\| {#1} \right\|}}

\newtheorem{theorem}{Theorem}
\newtheorem*{theorem*}{Theorem}
\newtheorem{lemma}[theorem]{Lemma}

\newtheorem{klemma}[theorem]{Key Lemma}
\newtheorem{corollary}[theorem]{Corollary}
\newtheorem*{corollary*}{Corollary}
\newtheorem{proposition}[theorem]{Proposition}

\newtheorem{definition}[theorem]{Definition}

\newtheorem{remark}[theorem]{Remark}

\numberwithin{subsection}{section}
\numberwithin{theorem}{section}
\numberwithin{equation}{section}
\numberwithin{figure}{section}

\usepackage[affil-it]{authblk}

\title{A T(1) theorem for fractional Sobolev spaces on domains}

\author{Mart\'i Prats and Eero Saksman
\thanks{MP (Departament de Ma\-te\-m\`a\-ti\-ques, Universitat Aut\`onoma de Bar\-ce\-lo\-na, Catalonia): \texttt{mprats@mat.uab.cat}.
ES (Department of Mathematics and Statistics, University of Helsinki, Finland): \texttt{eero.saksman@helsinki.fi}.}}

\begin{document}
\maketitle
\bibliographystyle{alpha}

\begin{abstract} 
Given any uniform domain $\Omega$, the Triebel-Lizorkin space $F^s_{p,q}(\Omega)$ with $0<s<1$ and $1<p,q<\infty$ can be equipped with a norm in terms of first order differences restricted to pairs of points whose distance is comparable to their distance to the boundary. 

Using this characterization, originally due to Seeger and reproven here (see Remark \ref{remSeeger} below) we prove a T(1)-theorem for fractional Sobolev spaces with $0<s<1$ for any uniform domain and for a large family of Calder\'on-Zygmund operators in any ambient space $\mathbb{R}^d$ as long as $sp>d$.
\end{abstract}

\section{Introduction}
The aim of the present article is to find necessary and sufficient conditions on certain singular integral operators to be bounded in fractional Sobolev spaces of a uniform domain $\Omega$ with smoothness $0<s<1$. However, the results are valid in $F^s_{p,q}(\Omega)$, that is, the so-called Triebel-Lizorkin spaces, when $s>\max\left\{0,\frac{d}{p}-\frac{d}{q}\right\}$.

Consider $0 < \sigma\leq 1$. An operator $T$ defined for $f\in L^1_{loc}(\R^d)$ and $x\in \R^d\setminus \supp(f)$ as 
$$Tf(x) = \int_{\R^d} K(x-y)f(y) dy,$$
is called a \emph{convolution Calder\'on-Zygmund operator of order $\sigma$} if it is bounded on $L^p(\R^d)$ for every $1<p<\infty$ and its kernel $K$ satisfies the size condition
\begin{equation*}
|K(x)|\leq \frac{ C_K}{|x|^{d}} \mbox{\quad\quad for every }x\neq 0
\end{equation*}
and the H\"older smoothness condition
\begin{equation*}
|K(x-y)-K(x)|\leq \frac{ C_K|y|^\sigma}{|x|^{d+\sigma}} \mbox{\quad\quad for every }0<2|y|\leq |x|
\end{equation*}
 (see Section \ref{secT1} for more details). In the present article we deal with some properties of the operator $T$ truncated to a domain $\Omega$, defined as $T_\Omega(f)=\chi_\Omega \, T(\chi_\Omega \, f)$. 
 
In the complex plane, for instance, the \emph{Beurling transform}, which is defined as the principal value 
\begin{equation*}
Bf(z):=-\frac{1}{\pi}\lim_{\varepsilon\to0}\int_{|w-z|>\varepsilon}\frac{f(w)}{(z-w)^2}dm(w),
\end{equation*}
is a convolution Calder\'on-Zygmund operator of any order with kernel
$K(z)=-\frac1{\pi \, z^2}.$

In the article \cite{CruzMateuOrobitg}, V\'ictor Cruz, Joan Mateu and Joan Orobitg, seeking for some results on the Sobolev smoothness of quasiconformal mappings proved the next theorem.
\begin{theorem*}[see \cite{CruzMateuOrobitg}]
Let $\Omega\subset\R^d$ be a bounded $C^{1+\varepsilon}$ domain (i.e. a Lipschitz domain with parameterizations of the boundary in $C^{1+\varepsilon}$) for a given $\varepsilon>0$, and let $1<p<\infty$ and $0<s\leq1$ such that $sp>2$. Then any truncated Calder\'on-Zygmund operator $T_\Omega$ with smooth, homogeneous and even kernel is bounded in the Sobolev space $W^{s,p}(\Omega)$ if and only if $T(\chi_\Omega)\in W^{s,p}(\Omega)$.
\end{theorem*}

Later, Xavier Tolsa and the first author of the present paper, studied the case $s\in \N$, finding the following $T(P)$ Theorem.
\begin{theorem*}[see \cite{PratsTolsa}]
Let $\Omega\subset \R^d$ be a Lipschitz domain, $T$ a convolution Calder\'on-Zygmund operator with kernel $K$ satisfying 
$$|\nabla^j K(x)|\leq C \frac{1}{|x|^{d+j}} \mbox{\quad\quad for all } 0\leq j\leq n,\, x\neq 0,$$
 and $p>d$. Then the following statements are equivalent:
\begin{enumerate}[a)]
\item The truncated operator $T_\Omega$ is bounded in $W^{n,p}(\Omega)$.
\item For every polynomial $P$ of degree $n-1$, we have that $T_\Omega(P)\in W^{n,p}(\Omega)$.
\end{enumerate}
\end{theorem*}
Note that the kernels are not assumed to be even, and the conditions on the smoothness of the domain are relaxed. The authors assert that the theorem is valid even for uniform domains.

In the present paper we study again the fractional smoothness, but we deal with the case of uniform domains  (see Section \ref{secUniform}) for Triebel-Lizorkin spaces $F^s_{p,q}$ with $1<p,q<\infty$, $\max \left\{ 0, \frac{d}{p}-\frac{d}{q}\right\}<s<1$. 
 Let us note here to illustrate that in case $q=2$ we deal with the Sobolev fractional spaces $W^{s,p}$ and in case $q=p$ then we deal with the Besov spaces $B^s_{p,p}$. To avoid misunderstandings, the reader must be aware that the $B^s_{p,p}$ spaces are called also Sobolev spaces in some books, while the $W^{s,p}$ spaces are sometimes called Bessel potential spaces. See Section \ref{secSobolev} for all the definitions of these spaces.
 
 Our main result is the following.
\begin{theorem}\label{theoT1}
Let $\Omega\subset \R^d$ be a bounded uniform domain, $T$ a convolution Calder\'on-Zygmund operator of order $0<s< 1$. Consider indices $p,q\in(1,\infty)$ with $s>\frac{d}{p}$. Then 
the truncated operator $T_\Omega$ is bounded in $F^s_{p,q}(\Omega)$ if and only if we have that $T_\Omega(1)\in F^s_{p,q}(\Omega)$.
\end{theorem}

To prove this result we will need an equivalent norm for $F^s_{p,q}$. The following result is not present in the literature in its full generality, but it is found for the Sobolev case in \cite{SteinCharacterization} and for the general Triebel-Lizorkin case when $s>\frac{d}{\min\{p,q\}}$ in \cite[Theorem 2.5.10]{TriebelTheory}. The result as stated below will be  a corollary of some results in \cite{TriebelTheoryIII}.
\begin{theorem}[see Corollary \ref{coroCharacterizationRd}]\label{theoNormRd}
Let $1\leq p<\infty$, $1\leq q\leq \infty$ and $0<s< 1$ with $s>\frac{d}{p}-\frac{d}{q}$. 
Then,
\begin{equation*}
F^s_{p,q}
	 =\left\{f\in L^{\max\{p,q\}}:  \norm{f}_{L^p}+\left(\int_{\R^d} \left(\int_{\R^d}\frac{|f(x)-f(y)|^q}{|x-y|^{sq+d}} \,dy\right)^{\frac{p}{q}}dx\right)^{\frac{1}{p}}< \infty\right\},
\end{equation*}
	 (with the usual modification for $q=\infty$), in the sense of equivalent norms.
\end{theorem}
The restriction  $s>\frac{d}{p}-\frac{d}{q}$ is sharp, as we will see in Remark \ref{remsdpdq}. One can find some equivalent norms for Triebel-Lizorkin spaces in terms of differences using means on balls which avoid this restriction. We refer the reader to \cite{Strichartz} or \cite[Corollary 2.5.11]{TriebelTheory}.

Given a domain $\Omega$ and a locally integrable function $f$, we say that $f\in F^s_{p,q}(\Omega)$ if there is a function $h\in F^s_{p,q}(\R^d)$ such that $h|_\Omega=f|_\Omega$. The norm $\norm{f}_{F^s_{p,q}(\Omega)}$ will be defined as the infimum of the norms $\norm{h}_{F^s_{p,q}(\R^d)}$ for all admissible $h$. Our method is based on an intrinsic characterization of this norm, inspired by the previous theorem. We define
\begin{equation*}
\norm{f}_{{A}^s_{p,q}(\Omega)}:= \norm{f}_{L^p(\Omega)} + \left(\int_\Omega \left(\int_{\Omega}\frac{|f(x)-f(y)|^q}{|x-y|^{sq+d}} \,dy\right)^{\frac{p}{q}}dx\right)^{\frac{1}{p}}.
\end{equation*}
Indeed, this norm will be equivalent to the Triebel-Lizorkin one for uniform domains:
\begin{theorem}\label{theoNormOmegaEquivalent}
Let $\Omega\subset \R^d$ be a bounded uniform domain, $1<p,q<\infty$ and $0<s< 1$ with $s>\frac{d}{p}-\frac{d}{q}$. 
Then $f\in F^s_{p,q}(\Omega)$ if and only if $f\in A^s_{p,q}(\Omega)$ and the norms are equivalent.
\end{theorem}
To prove this result we will use Theorem \ref{theoNormRd} and the following extension Theorem:
\begin{theorem}\label{theoExtension}
Let $\Omega\subset \R^d$ be a  bounded uniform domain, $1<p,q<\infty$ and $0<s< 1$ with $s>\frac{d}{p}-\frac{d}{q}$. 
Then there exists a bounded operator $\Lambda_0: A^s_{p,q}(\Omega)\to F^s_{p,q}(\R^d)$ such that $\Lambda_0 f|_\Omega=f$ for every $f\in A^s_{p,q}(\Omega)$.
\end{theorem}

However, in the proof of Theorem \ref{theoT1} we will make use of a functional which is closely related to $\norm{\cdot}_{A^s_{p,q}(\Omega)}$. Call $\delta(x)=\dist(x,\partial\Omega)$. Consider the Carleson boxes (or shadows) $\mathbf{Sh}(x):=\{y\in\Omega : |y-x|\leq c_{\Omega}\delta(x)\}$ with $c_\Omega>1$ to be fixed (see Section \ref{secUniform}). Then we have the following reduction for the Triebel-Lizorkin norm:
\begin{theorem}[See Corollary \ref{coroExtensionDomainFinal}.]\label{theoNormOmega}
Let $\Omega\subset \R^d$ be a bounded uniform domain, $1<p<q<\infty$ and $0<s< 1$ with $s>\frac{d}{p}-\frac{d}{q}$. 
Then $f\in F^s_{p,q}(\Omega)$ if and only if 
\begin{equation*}
\norm{f}_{L^p(\Omega)} + \left(\int_\Omega \left(\int_{\Sh(x)}\frac{|f(x)-f(y)|^q}{|x-y|^{sq+d}} \,dy\right)^{\frac{p}{q}}dx\right)^{\frac{1}{p}}<\infty.
\end{equation*}
Furthermore, the left-hand side of the inequality above is equivalent to the norm $\norm{f}_{F^s_{p,q}(\Omega)}$.
\end{theorem}

The situation is even better when $p\geq q$:
\begin{theorem}[See Corollary \ref{coroExtensionDomainFinal}.]\label{theoNormOmegapgtrq}
Let $\Omega\subset \R^d$ be a bounded uniform domain, $1<q\leq p <\infty$,  $0<s< 1$ and $0<\rho<1$. 
Then $f\in F^s_{p,q}(\Omega)$ if and only if 
\begin{equation*}
\norm{f}_{L^p(\Omega)} + \left(\int_\Omega \left(\int_{B\left(x,\rho \delta(x)\right)}\frac{|f(x)-f(y)|^q}{|x-y|^{sq+d}} \,dy\right)^{\frac{p}{q}}dx\right)^{\frac{1}{p}}<\infty.
\end{equation*}
Furthermore, the left-hand side of the inequality above is equivalent to the norm $\norm{f}_{F^s_{p,q}(\Omega)}$.
\end{theorem}

In particular, for every $1<p<\infty$, $0<s<1$ and $0<\rho<1$, we have that 
$$\norm{f}_{B^s_{p,p}(\Omega)}\approx \norm{f}_{L^p(\Omega)}+\left(\int_\Omega \int_{B_{\rho\delta(x)}(x)} \frac{| f(x)- f(y)|^p}{|x-y|^{sp+d}} \,dy  \,dx\right)^{\frac{1}{p}} \mbox{\quad\quad for all } f\in B^s_{p,p}(\Omega).$$

If in addition $p\geq 2$ we have that $$\norm{f}_{W^{s,p}(\Omega)}\approx \norm{f}_{L^p(\Omega)}+\left(\int_\Omega\left( \int_{B_{\rho\delta(x)}(x)} \frac{| f(x)- f(y)|^2}{|x-y|^{2s+d}} \,dy \right)^\frac{p}{2} \,dx\right)^{\frac{1}{p}} \mbox{\quad\quad for all } f\in W^{s,p}(\Omega),$$
 and, if $1<p<2$  with $s > \frac{d}{p}-\frac{d}{2}$, we have that
$$\norm{f}_{W^{s,p}(\Omega)}\approx \norm{f}_{L^p(\Omega)}+\left(\int_\Omega\left( \int_{\Sh(x)} \frac{| f(x)- f(y)|^2}{|x-y|^{2s+d}} \,dy \right)^\frac{p}{2} \,dx\right)^{\frac{1}{p}} \mbox{\quad\quad for all } f\in W^{s,p}(\Omega).$$

The plan of the paper is the following. In Section \ref{secUniform} we define uniform domains in the spirit of \cite{Jones} but from a dyadic point of view and then we prove some basic properties of those domains. The expert reader may skip this part. Section \ref{secSobolev} begins with some remarks on Triebel-Lizorkin spaces, followed by the proof of the implicit characterization of Triebel-Lizorkin spaces given in Theorem  \ref{theoNormRd}, the Extension Theorem \ref{theoExtension} and, as a corollary, Theorem \ref{theoNormOmegaEquivalent}.  Section \ref{secNorms} is devoted to proving Theorems \ref{theoNormOmega} and \ref{theoNormOmegapgtrq} which are about the change of the domain of integration in the norm $A^s_{p,q}(\Omega)$. Section \ref{secT1} is the core of the paper, and it contains the proof of the T(1) Theorem \ref{theoT1}. The key Lemma \ref{lemKLemmaT1} is a discretization of the transform of a function and it is the cornerstone of the mentioned theorem. 

{\bf On notation:} 
When comparing two quantities $x_1$ and $x_2$ that depend on some parameters $p_1,\dots, p_j$ we will write 
$$x_1\leq C_{p_{i_1},\dots, p_{i_j}} x_2$$
if the constant $C_{p_{i_1},\dots, p_{i_j}} $ depends on ${p_{i_1},\dots, p_{i_j}}$. We will also write $x_1\lesssim_{p_{i_1},\dots, p_{i_j}} x_2$ for short, or simply $x_1\lesssim  x_2$ if the dependence is clear from the context or if the constants are universal. We may omit some of these variables for the sake of simplicity. The notation $x_1 \approx_{p_{i_1},\dots, p_{i_j}} x_2$ will mean that $x_1 \lesssim_{p_{i_1},\dots, p_{i_j}} x_2$ and $x_2 \lesssim_{p_{i_1},\dots, p_{i_j}} x_1$.

Given a cube $Q$, we write $\ell(Q)$ for its side-length. 
Given two cubes $Q,S$, we define their long distance as
$\Dist(Q,S)=\ell(Q)+\dist(Q,S)+\ell(S)$.
Given a real number $\rho$, we define $\rho Q$ as the cube concentric to $Q$, with ratio $\rho$ and faces parallel to the faces of $Q$.

For any cube $Q$ and any function $f$, we call $f_Q= \fint_Q f \, dm$ to the mean of $f$ in $Q$.

Given $1\leq p\leq \infty$ we write $p'$ for its H\"older conjugate, that is $\frac{1}{p}+\frac{1}{p'}=1$.

\begin{remark}\label{remSeeger}
Long after the present paper was published we learned that our Theorem \ref{theoNormOmegapgtrq} can be obtained also as a corollary of Seeger's result  \cite[Corollary 1]{Seeger} in a more general framework. Namely, choosing  $A_t=t Id$ and $s<1$ in  \cite[Corollary 1]{Seeger} one gets the reduction described in Theorem \ref{theoNormOmegapgtrq}: 

Let $\Omega\subset \R^d$ be a bounded uniform domain, $1< p<\infty$, $1< q\leq \infty$,  $0<s< 1$ with $s>\frac dp-\frac dq$ and $0<\rho<1$. 
Then $f\in F^s_{p,q}(\Omega)$ if and only if 
\begin{equation*}
\norm{f}_{L^p(\Omega)} + \left(\int_\Omega \left(\int_{B\left(x,\rho \delta(x)\right)}\frac{|f(x)-f(y)|^q}{|x-y|^{sq+d}} \,dy\right)^{\frac{p}{q}}dx\right)^{\frac{1}{p}}<\infty.
\end{equation*}
Furthermore, the left-hand side of the inequality above is equivalent to the norm $\norm{f}_{F^s_{p,q}(\Omega)}$.

Note that Seeger's result is valid in the range $1<p<\infty$ and $1<q\leq \infty$ and $0<s<1$ assuming that $s>\frac dp -\frac dq$, which is greater than the range in Theorem \ref{theoNormOmegapgtrq}. In particular, the result improves Theorem \ref{theoNormOmega} in all its range of indices. One can modify the proof of the present paper to obtain the same result, available under personal communication with the authors. In fact, the range can be extended to $1\leq p <\infty$ and $1\leq q \leq \infty$, see the appendix of \cite{PratsTL}. Since the present paper has already been published in the Journal of Geometric Analysis, we keep the original typing also here to preserve the uniqueness of the references. 
\end{remark}

\section{On uniform domains}\label{secUniform}
There is a considerable literature on uniform domains and their properties, we refer the reader e.g. to \cite{GehringOsgood} and \cite{Vaisala}.

\begin{definition}\label{defWhitney}
Given a domain $\Omega$, we say that a collection of open dyadic cubes $\mathcal{W}$ is a {\rm Whitney covering} of $\Omega$ if they are disjoint, the union of the cubes and their boundaries is $\Omega$, there exists a constant $C_{\mathcal{W}}$ such that 
$$C_\mathcal{W} \ell(Q)\leq \dist(Q, \partial\Omega)\leq 4C_\mathcal{W}\ell(Q),$$
two neighbor cubes $Q$ and $R$ (i.e., $\overline Q\cap \overline R\neq\emptyset$) satisfy $\ell(Q)\leq 2 \ell(R)$, and the family $\{50 Q\}_{Q\in\mathcal{W}}$ has finite superposition. Moreover, we will assume that 
\begin{equation}\label{eqWhitney5}
S\subset 5Q \implies \ell(S)\geq \frac12 \ell(Q).
\end{equation}
\end{definition}
The existence of such a covering is granted for any open set different from $\R^d$ and in particular for any domain as long as $C_\mathcal{W}$ is big enough (see \cite[Chapter 1]{SteinPetit} for instance).

\begin{figure}[ht]
 \centering
 \includegraphics[width=0.7\textwidth]{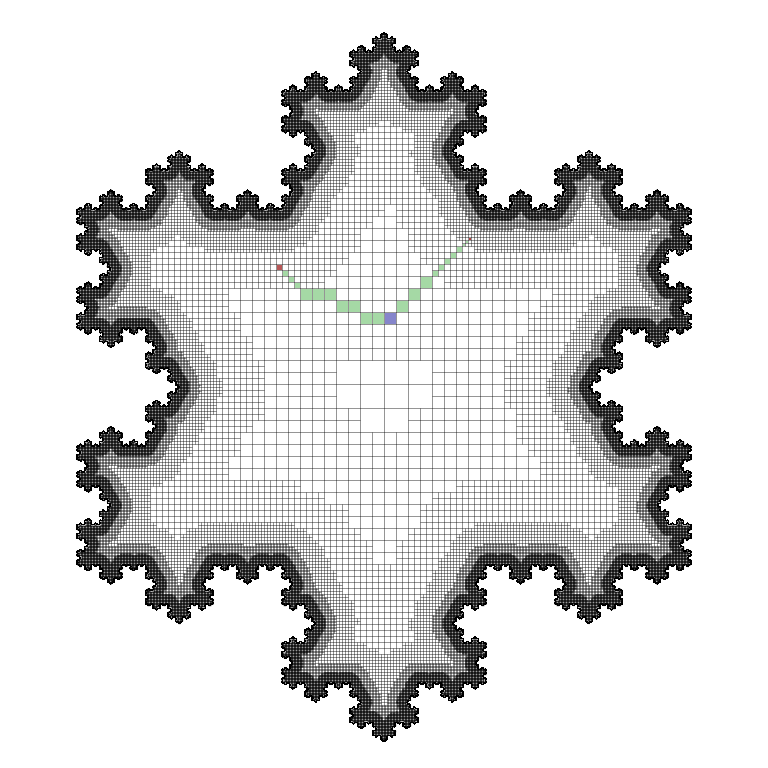}
  \caption{A Whitney decomposition of a uniform domain with and an $\varepsilon$-admissible chain. The end-point cubes are colored in red and the central one in blue.}\label{figCovering}
\end{figure}

\begin{definition}\label{defEpsilonAdmissible}
Let $\Omega$ be a domain, $\mathcal{W}$ a Whitney decomposition of $\Omega$ and $Q,S\in\mathcal{W}$.  Given $M$ cubes $Q_1,\dots,Q_M\in\mathcal{W}$ with $Q_1=Q$ and $Q_M=S$, the $M$-tuple $(Q_1,\dots,Q_M)_{j=1}^M\in\mathcal{W}^M$  is a {\em chain} connecting $Q$ and $S$ if the cubes $Q_j$ and $Q_{j+1}$ are neighbors for $j<M$. We write $[Q,S]=(Q_1,\dots,Q_M)_{j=1}^M$ for short.

Let $\varepsilon\in\R$. We say that the chain $[Q,S]$ is {\em $\varepsilon$-admissible} if 
\begin{itemize}
\item the \emph{length}  of the chain is bounded by
\begin{equation}\label{eqLengthDistance}
\ell([Q,S]):=\sum_{j=1}^M\ell(Q_j)\leq \frac1\varepsilon\Dist(Q,S)
\end{equation}
\item and there exists $j_0<M$ such that the cubes in the chain satisfy
\begin{equation}\label{eqAdmissible1}
\ell(Q_j)\geq\varepsilon \Dist(Q_1,Q_j) \mbox{ for all } j\leq j_0 \mbox{\quad\quad  and \quad\quad }
\ell(Q_j)\geq\varepsilon \Dist(Q_j,Q_M) \mbox{ for all } j\geq j_0 .
\end{equation}
\end{itemize}
The $j_0$-th cube, which we call \emph{central}, satisfies that $\ell(Q_{j_0})\gtrsim_d \varepsilon \Dist(Q,S)$ by \rf{eqAdmissible1} and the triangle inequality. We will write  $Q_S=Q_{j_0}$. Note that this is an abuse of notation because the central cube of $[Q,S]$ may vary for different $\varepsilon$-admissible chains joining $Q$ and $S$.

We write (abusing notation again) $[Q,S]$ also for the set $\{Q_j\}_{j=1}^M$. Thus, we will write $P\in[Q,S]$ if $P$ appears in a coordinate of the $M$-tuple $[Q,S]$.
 For any $P\in [Q,S]$ we call $\mathcal{N}_{[Q,S]}(P)$ to the following cube in the chain, that is, for $j<M$ we have that $\mathcal{N}_{[Q,S]}(Q_j)=Q_{j+1}$. We will write $\mathcal{N}(P)$ for short if the chain to which we are referring is clear from the context.

Every now and then we will mention subchains. That is, for $1\leq j_1\leq j_2\leq M$, the subchain $[Q_{j_1},Q_{j_2}]_{[Q,S]} \subset[Q,S]$ is defined as $(Q_{j_1},Q_{j_1+1},\dots,Q_{j_2})$. We will write $[Q_{j_1},Q_{j_2}]$ if there is no risk of confusion.
\end{definition}

Next we make some observations on the two subchains $[Q,Q_S]$ and $[Q_S,S]$.
\begin{remark}\label{remDistances}
Consider a domain $\Omega$ with covering $\mathcal{W}$ and two cubes $Q,S\in\mathcal{W}$ with an $\varepsilon$-admissible chain $[Q,S]$. From Definition \ref{defEpsilonAdmissible} it follows that
\begin{equation}\label{eqAdmissible2}
\Dist(Q,S)\approx_{\varepsilon,d} \ell([Q,S])\approx_{\varepsilon,d} \ell(Q_S) \approx_{\varepsilon,d} \Dist(Q,Q_S)\approx_{\varepsilon,d} \Dist(Q_S,S).
\end{equation}

If $P\in[Q,Q_S]$, by  \rf{eqAdmissible1} we have that 
\begin{equation}\label{eqDistanceBoundedByLP} 
\Dist(Q,P)\approx_{d,\varepsilon} \ell(P).
\end{equation}
On the other hand,  by the triangular inequality, \rf{eqLengthDistance} and \rf{eqAdmissible1} we have that
\begin{equation*}
\Dist(P,S) \lesssim_d \ell([P,S]) \leq \ell([Q,S]) \leq\frac{\Dist(Q,S)}{\varepsilon}\lesssim_d \frac{\Dist(Q,P)+\Dist(P,S)}{\varepsilon}\lesssim_d \frac{\frac{1}{\varepsilon}\ell(P)+\Dist(P,S)}{\varepsilon} ,
\end{equation*}
that is,
\begin{equation}\label{eqNotVeryFar}
\Dist(P,S)\approx_{\varepsilon,d} \Dist(Q,S).
\end{equation}
\end{remark}

\begin{definition}\label{defUniform}
We say that a domain $\Omega\subset\R^d$ is a {\em uniform domain} if there exists a Whitney covering $\mathcal{W}$ of $\Omega$ and $\varepsilon \in\R$ such that for any pair of cubes $Q,S \in\mathcal{W}$, there exists an $\varepsilon$-admissible chain $[Q,S]$ (see Figure \ref{figCovering}). Sometimes  will write {\em $\varepsilon$-uniform domain} to fix the constant $\varepsilon$.
\end{definition}

Using \rf{eqNotVeryFar} it is quite easy to see that a domain satisfying this definition satisfies to the one given by Peter Jones in \cite{Jones} with $\delta=\infty$ (changing the parameter $\varepsilon$ if necessary). It is somewhat more involved to prove the converse implication, but it can be done using the ideas of Remark \ref{remDistances}. In any case it is not transcendent for the present paper to prove this fact, which is left for the reader as an exercise.

Now we can define the shadows:
\begin{definition}\label{defShadow}
Let $\Omega$ be an $\varepsilon$-uniform domain with Whitney covering $\mathcal{W}$. 
Given a cube $P\in\mathcal{W}$ centered at $x_P$ and a real number $\rho$,  the {\em $\rho$-shadow} of $P$ is the collection of cubes
$$\SH_\rho(P)=\{Q\in\mathcal{W}:Q\subset B(x_P,\rho\,\ell(P))\}, $$
and its  {\em ``realization''} is the set
$$\Sh_{\rho}(P)=\bigcup_{Q\in\SH_\rho(P)} Q$$
(see Figure \ref{figShadow}). 

By the previous remark and the properties of the Whitney covering, we can define $\rho_\varepsilon>1$ such that the following properties hold:
\begin{itemize}
\item  For every $P\in\mathcal{W}$, we have the estimate $|\diam(\partial\Omega\cap\overline{\Sh_{\rho_\varepsilon}(P)})|\approx \ell(P)$.
\item For every $\varepsilon$-admissible chain $[Q,S]$, and every $P\in[Q,Q_S]$ we have that $Q\in\SH_{\rho_\varepsilon}(P)$.
\item Moreover, every cube $P$ belonging to an $\varepsilon$-admissible chain $[Q,S]$ belongs to the shadow $\SH_{\rho_\varepsilon}(Q_S)$.
\end{itemize}
\end{definition}

\begin{figure}[ht]
 \centering
  \includegraphics[width=0.5\textwidth]{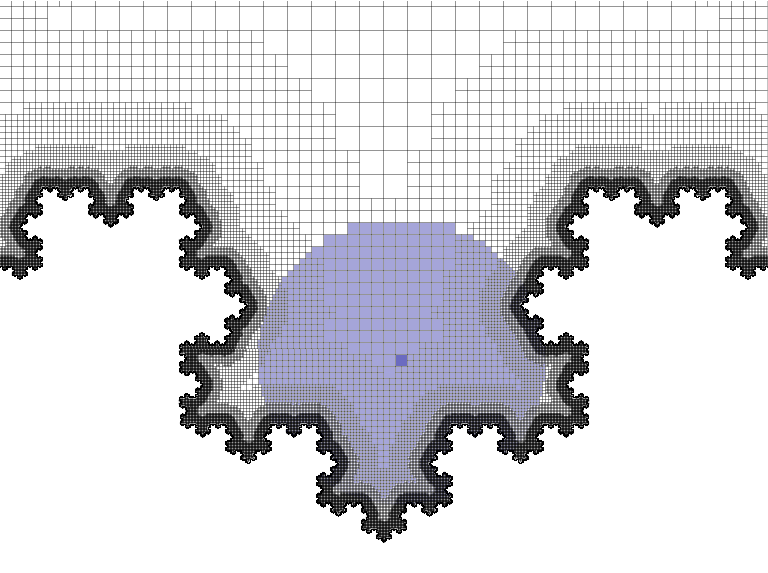}
  \caption{The shadow  $\Sh_{13}(P)$.}\label{figShadow}
\end{figure}

Note that the first property comes straight from the properties of the Whitney covering, while the second is a consequence of \rf{eqDistanceBoundedByLP} and the third holds because every cube  $P$ contained in the chain $[Q,S]$ satisfies  $D(P,Q_S)\lesssim_d \ell([Q,S])\approx\Dist(Q,S)\approx\ell(Q_S)$ by \rf{eqAdmissible2}. 

\begin{remark}\label{remInTheShadow}
Given an $\varepsilon$-uniform domain $\Omega$ we will write $\Sh$ for $\Sh_{\rho_\varepsilon}$. We will write also $\SH$ for $\SH_{\rho_{\varepsilon}}$.

For $Q\in\mathcal{W}$ and $s>0$,  we have that  
\begin{equation}\label{eqAscendingToGlory}
 \sum_{L: Q\in \SH(L)}\ell(L)^{-s} \lesssim \ell(Q)^{-s} 
 \end{equation}
and, moreover, if $Q\in\SH(P)$, then
\begin{equation}\label{eqAscendingPath}
 \sum_{L\in[Q,P]}\ell(L)^{s} \lesssim \ell(P)^s \mbox{\quad\quad and \quad\quad}  \sum_{L\in[Q,P]}\ell(L)^{-s}\lesssim \ell(Q)^{-s} .
 \end{equation}
\end{remark}

\begin{proof}
Considering the definition of shadow we can deduce that there is a bounded number of cubes with given side-length in the left-hand side of \rf{eqAscendingToGlory}  and, therefore, the sum is a geometric sum. Again by the definition of shadow we know that the smaller cube in that  sum has side-length comparable to $\ell(Q)$.

To prove \rf{eqAscendingPath}, first note that $\ell(Q_P)\approx \Dist(Q,P)\approx \ell(P)$ by \rf{eqAdmissible2} and Definition \ref{defShadow}. For every $L\in [Q,P]$, although it may occur that $L\notin\SH(P)$, we still have that by the triangle inequality $\Dist(L,P)\lesssim \ell([Q,P])\approx \Dist(Q,P)$ and, thus, by the definition of shadow we have that $\Dist(L,P)\lesssim\ell(P)$, i.e.
\begin{equation}\label{eqLChainInShadowP}
\Dist(L,P)\approx\ell(P).
\end{equation} 
When $L\in[Q,Q_P]$, \rf{eqDistanceBoundedByLP} reads as 
$$\ell(L)\approx \Dist(Q,L),$$
and when $L\in[Q_P,P]$ by \rf{eqDistanceBoundedByLP} and \rf{eqLChainInShadowP}, we have that
$$\ell(L)\approx\Dist(L,P)\approx \ell(P).$$
In particular, the number of cubes in $[Q_P,P]$ is uniformly bounded.  Summing up, for $L\in [Q,P]$ we have that $\ell(Q)\lesssim\ell(L)\lesssim\ell(P)$ and  all the cubes of a given side-length $r$ contained in $[Q,P]$ are situated at a distance from $Q$ bounded by $Cr$, so the number of those cubes is uniformly bounded. Therefore, the left-hand side of both inequalities in \rf{eqAscendingPath} are geometric sums, bounded by a constant times the bigger term. The constant depends on $s$, but also on the uniformity constant of the domain.
\end{proof}

We recall the definition of the non-centered Hardy-Littlewood maximal operator. Given $f\in L^1_{loc}(\R^d)$ and $x\in\R^d$, we define $Mf(x)$ as the supremum of the mean of $f$ in cubes containing $x$, that is,
$$Mf(x)=\sup_{Q:  x\in Q} \frac{1}{|Q|} \int_Q f(y) \, dy.$$
It is a well known fact that this operator is bounded in $L^p$ for $1<p<\infty$.
The following lemma is proven in \cite{PratsTolsa} and will be used repeatedly along the proofs contained in the present text.

\begin{lemma}\label{lemMaximal}
Let $\Omega$ be a bounded uniform domain with an admissible Whitney covering $\mathcal{W}$. Assume that $g\in L^1(\Omega)$ and $r>0$. For every $\eta>0$, $Q\in\mathcal{W}$ and $x\in \R^d$, we have
\begin{enumerate}[1)]
\item The non-local inequality for the maximal operator
	\begin{equation}\label{eqMaximalFar}
	 \int_{|y-x|>r} \frac{g(y) \, dy}{|y-x|^{d+\eta}}\lesssim_d \frac{Mg(x)}{r ^\eta}
\mbox{\quad\quad and \quad\quad}
	 \sum_{S:\Dist(Q,S)>r}  \frac{\int_S g(y) \, dy}{D(Q,S)^{d+\eta}}\lesssim_d \frac{\inf_{y\in Q} Mg(y)}{r ^\eta}.
	 \end{equation}
\item The local inequality for the maximal operator
	\begin{equation}\label{eqMaximalClose}
	 \int_{|y-x|<r} \frac{g(y) \, dy}{|y-x|^{d-\eta}}\lesssim_d r ^\eta Mg(x)
\mbox{\quad\quad and \quad\quad}
	\sum_{S:\Dist(Q,S)<r}  \frac{\int_S g(y) \, dy}{D(Q,S)^{d-\eta}}\lesssim_d \inf_{y\in Q} Mg(y) \,r^\eta.
	 \end{equation}
\item In particular we have
	\begin{equation}\label{eqMaximalAllOver}
		\sum_{S\in\mathcal{W}} \frac{\ell(S)^d}{\Dist(Q,S)^{d+\eta}} \lesssim_d \frac{1}{\ell(Q)^\eta}
	\end{equation}
and, by Definition \ref{defShadow},
	\begin{equation*}
	\sum_{S\in\SH_\rho(Q)} \int_S g(x) \, dx\lesssim_{d,\rho} \inf_{y\in Q} Mg(y) \, \ell(Q)^d.
	 \end{equation*}
\end{enumerate}
\end{lemma}

\section{Fractional Sobolev spaces}\label{secSobolev}
First we recall some results on Triebel-Lizorkin spaces. We refer the reader to \cite{TriebelTheory}.

\begin{definition}\label{defCollection}Let $\Phi(\R^d)$ be the collection of all the families of smooth functions $\Psi=\{\psi_j\}_{j=0}^\infty\subset C^\infty_c(\R^d)$ such that
\begin{equation*}
\left\{ 
\begin{array}{ll}
\supp \,\psi_0 \subset \DDD(0,2), & \\
\supp \,\psi_j \subset \DDD(0,2^{j+1})\setminus \DDD(0,2^{j-1}) & \mbox{ if $j\geq 1$},\\	
\end{array}
\right.
\end{equation*}
for every multiindex $\alpha\in \N^d$ there exists a constant $c_\alpha$ such that
\begin{equation*}
\norm{D^\alpha \psi_j}_\infty \leq \frac{c_\alpha}{2^{j |\alpha|} } \mbox{\,\,\, for every $j\geq 0$}
\end{equation*}
and
\begin{equation*}
\sum_{j=0}^\infty \psi_j(x)=1 \mbox{\,\,\, for every $x\in\R^d$.}
\end{equation*}
\end{definition}

We will use the classical notation $\widehat f$ for the Fourier transform of a given Schwartz function,
$$\widehat f (\xi)=\int_{\R^d} e^{-2\pi i x\cdot \xi} f(x)\, dx,$$ 
and $\widecheck f$ will denote its inverse.
It is well known that the Fourier transform can be extended to the whole space of tempered distributions by duality and it induces an isometry in $L^2$ (see for example \cite[Chapter 2]{Grafakos}).

 \begin{definition}
Let $s \in \R$,  $1\leq p\leq \infty$, $1\leq q\leq\infty$ and $\Psi \in \Phi(\R^d)$. For any tempered distribution $f\in \mathcal{S}'(\R^d)$ we define its {\em non-homogeneous Besov norm}
\begin{equation*}
\norm{f}_{B^s_{p,q}}^\Psi=\norm{\left\{2^{sj}\norm{\left(\psi_j \widehat{f}\right)\widecheck{\,}\, }_{L^p}\right\}}_{l^q},
\end{equation*}
and we call $B^s_{p,q}\subset \mathcal{S}'$ to the set of tempered distributions such that this norm is finite.

Let $s \in \R$,  $1\leq p< \infty$, $1\leq q\leq\infty$ and $\Psi \in \Phi(\R^d)$. For any tempered distribution $f\in \mathcal{S}'(\R^d)$ we define its {\em non-homogeneous Triebel-Lizorkin norm}
\begin{equation*}
\norm{f}_{F^s_{p,q}}^\Psi=\norm{\norm{\left\{2^{sj}\left(\psi_j \widehat{f}\right)\widecheck{\,}\right\}}_{l^q}}_{L^p},
\end{equation*}
and we call $F^s_{p,q}\subset \mathcal{S}'$ to the set of tempered distributions such that this norm is finite. 
\end{definition}

These norms are equivalent for different choices of $\Psi$. Of course we will omit $\Psi$ in our notation since it plays no role (see \cite[Section 2.3]{TriebelTheory}).

\begin{remark}
For $q=2$ and $1<p<\infty$ the spaces $F^s_{p,2}$ coincide with the so-called Bessel-potential spaces $W^{s,p}$. In addition, if $s\in \N$ they coincide with the usual Sobolev spaces of functions in $L^p$ with weak derivatives up to order $s$ in $L^p$,  and they coincide with $L^p$ for $s=0$ (\cite[Section 2.5.6]{TriebelTheory}). In the present text, we call Sobolev space to any $W^{s,p}$ with $s>0$ and $1<p<\infty$, even if $s$ is not a natural number. Note that complex interpolation between Sobolev spaces is a Sobolev space (see \cite[Section 2.4.2, Theorem 1]{TriebelInterpolation}).
\end{remark}

To use the Sobolev embedding for Triebel-Lizorkin spaces, we will use the following proposition. 
\begin{proposition}[{See \cite[Section 2.3.2]{TriebelTheory}.}]\label{propoPropertiesBesovTriebel}

Let $1\leq q \leq \infty$ and $1\leq p<  \infty$, $s\in\R$ and $\varepsilon>0$. Then 
\begin{equation}\label{eqEmbedSmoothness}
F^{s+\varepsilon}_{p,q}\subset W^{s,p}.
\end{equation}
\end{proposition}

Next we will prove Theorem \ref{theoNormRd}. Let us write
$\Delta_h^1 f(x):=f(x+h)-f(x)$ and, if $M\in\N$ with $M>1$ we define the $M$-th iterated difference as $\Delta_h^Mf(x):=\Delta_h^1 (\Delta_h^{M-1} f)(x) = \sum_{j=0}^M{M\choose j} (-1)^{M-j} f(x+jh)$. Given $f\in L^1_{loc}$, an index $0< u\leq \infty$ and $t\in \R$, we write
$$d_{t,u}^Mf(x):=\left(t^{-d}\int_{|h|\leq t}|\Delta_h^M f(x)|^u\, dh \right)^\frac{1}{u},$$
with the usual modification for $u=\infty$. In  \cite[Theorem 1.116]{TriebelTheoryIII} we find the following result.
\begin{theorem*}[See \cite{TriebelTheoryIII}.]
Given $1\leq r \leq \infty$, $0<u\leq r$, $1\leq p<\infty$, $1\leq q\leq\infty$ and $0<s<M$ with $\frac{d}{\min\{p,q\}}-\frac{d}{r}<s$, we have that
\begin{equation*}
F^s_{p,q}(\R^d)=\left\{f\in L^{\max\{p,r\}}:  \norm{f}_{L^p} + \left(\int_{\R^d} \left(\int_0^1  \frac{d_{t,u}^M f(x)^q}{t^{sq+1}} \, dt \right)^{\frac{p}{q}}dx\right)^{\frac{1}{p}}< \infty \right\}
\end{equation*}
(with the usual modification for $q=\infty$), in the sense of equivalent quasinorms.
\end{theorem*}

As an immediate consequence of this result, we get the following corollary.
\begin{corollary}\label{coroCharacterizationRd}
Let $1\leq p<\infty$, $1\leq q\leq \infty$ and $0<s<1\leq M$ with $s>\frac{d}{p}-\frac{d}{q}$. Then
\begin{align*}
F^s_{p,q}(\R^d)
	& =\left\{f\in L^{\max\{p,q\}} \mbox{ s.t. }  \norm{f}_{A^s_{p,q}(\R^d)}:=\norm{f}_{L^p}+ \left(\int_{\R^d} \left(\int_{\R^d}\frac{|\Delta_h^M f(x)|^q}{|h|^{sq+d}} \,dh\right)^{\frac{p}{q}}dx\right)^{\frac{1}{p}}< \infty\right\}
\end{align*}
(with the usual modification for $q=\infty$), in the sense of equivalent norms.
\end{corollary}
\begin{proof}
 
Let  $f \in L^{\max\{p,q\}}$. Choosing $q=u=r$ all the conditions  in the theorem above are satisfied. Therefore,
\begin{equation}\label{eqFirstEquivalentNormTriebel}
\norm{f}_{F^s_{p,q}(\R^d)}\approx  \norm{f}_{L^p} + \left(\int_{\R^d} \left(\int_0^1  \frac{d_{t,q}^Mf(x)^q}{t^{sq+1}} \, dt \right)^{\frac{p}{q}}dx\right)^{\frac{1}{p}} .
\end{equation}
Since $d_{t,q}^Mf(x)=\left(t^{-d}\int_{|h|\leq t}|\Delta_h^M f(x)|^q\, dh \right)^\frac{1}{q}$ for $x\in\R^d$, we can change the order of integration
to get that
\begin{align*}
 \int_{\R^d} \left(\int_0^1  \frac{d_{t,q}^Mf(x)^q}{t^{sq+1}} \, dt \right)^{\frac{p}{q}} \,dx
 	& = \int_{\R^d} \left(\int_{|h|\leq 1} \int_{1>t>|h|}  \frac{dt}{t^{sq+1+d}} |\Delta_h^M f(x)|^q\, dh \right)^{\frac{p}{q}} \,dx\\
 	& = \int_{\R^d} \left(\int_{|h|\leq 1} \frac{  |\Delta_h^M f(x)|^q}{sq+d}\left(\frac{1}{|h|^{sq+d}}-1\right)\, dh \right)^{\frac{p}{q}} \,dx.
\end{align*}
This shows that $ \norm{f}_{F^s_{p,q}(\R^d)}\lesssim \norm{f}_{A^s_{p,q}(\R^d)}$ and also that
\begin{align}\label{eqChainOnYou}
\int_{\R^d} \left(\int_{|h|<\frac{1}{2}}\frac{|\Delta_h^M f(x)|^q}{|h|^{sq+d}} \,dh\right)^{\frac{p}{q}} \,dx
 			&  \lesssim \int_{\R^d} \left(\int_0^1  \frac{d_{t,q}^Mf(x)^q}{t^{sq+1}} \, dt \right)^{\frac{p}{q}} \,dx \lesssim \norm{f}_{F^s_{p,q}(\R^d)}^p
\end{align}
by \rf{eqFirstEquivalentNormTriebel}. It remains to see that $\int_{\R^d} \left(\int_{|h|>\frac{1}{2}}\frac{|\Delta_h^M f(x)|^q}{|h|^{sq+d}} \,dh\right)^{\frac{p}{q}} \,dx \lesssim  \norm{f}_{F^s_{p,q}(\R^d)}^p$. Using appropriate changes of variables and the triangle inequality, it is enough  to check that
\begin{equation}\label{eqTheSeparatedElements}
\circled{I}:=\int_{\R^d} \left(\int_{\R^d}\frac{|f(x+h)|^q}{(1+|h|)^{sq+d}} \,dh\right)^{\frac{p}{q}} \,dx \lesssim  \norm{f}_{F^s_{p,q}(\R^d)}^p .
\end{equation}

Let us assume first that $p\geq q$. Then, since the measure $(1+|h|)^{-(sq+d)}\, dh$ is finite, we may apply Jensen's inequality to the inner integral, and then Fubini to obtain
$$\circled{I} \lesssim \int_{\R^d} \int_{\R^d}\frac{|f(x+h)|^p}{(1+|h|)^{sp+d}} \,dh \,dx \lesssim \norm{f}_{L^p}^p, $$
and \rf{eqTheSeparatedElements} follows.

If, instead, $p<q$, cover $\R^d$ with disjoint cubes $Q_{\vec{j}}=Q_0+\ell \vec{j}$ for $\vec{j}\in\Z^d$. Fix the side-length $\ell$ of these cubes so that their diameter is $1/3$. By the subadditivity of $x\mapsto |x|^{\frac{p}{q}}$, we have that
\begin{align*}
\circled{I}
	& \lesssim\sum_{\vec{k}} \int_{Q_{\vec{k}}}\sum_{\vec{j}} \left(\int_{Q_{\vec{j}}} \frac{|f(y)|^q}{(1+|x-y|)^{sq+d}} \,dy\right)^{\frac{p}{q}} \,dx \approx \sum_{\vec{j}} \left(\int_{Q_{\vec{j}}} |f(y)|^q \,dy\right)^{\frac{p}{q}} \sum_{\vec{k}} \frac{1}{(1+|\vec{j}-\vec{k}|)^{sp+\frac{dp}{q}}}.
\end{align*}
Since $s+\frac{d}{q} > \frac{d}{p}$, the last sum is finite and does not depend on $\vec{j}$.
By \rf{eqChainOnYou} we have that
\begin{align*} 
\circled{I}
	& \lesssim \sum_{\vec{j}} \left(\int_{Q_{\vec{j}}} |f(y)|^q \,dy\right)^{\frac{p}{q}} 
	 \lesssim \sum_{\vec{j}} \int_{Q_{\vec{j}}} \left(\int_{Q_{\vec{j}}} |f(y)-f(x)|^q \,dy\right)^{\frac{p}{q}}dx + \sum_{\vec{j}}\int_{Q_{\vec{j}}} \left(\int_{Q_{\vec{j}}} |f(x)|^q \,dy\right)^{\frac{p}{q}}\, dx \\
	& \lesssim \norm{f}_{F^s_{p,q}(\R^d)}^p.
\end{align*}
In the last step we have used that $\sum_{\vec{j}}\int_{Q_{\vec{j}}} \left(\int_{Q_{\vec{j}}} |f(x)|^q \,dy\right)^{\frac{p}{q}}\, dx\approx \norm{f}_{L^p}^p $ because all the cubes have side-length comparable to $1$, and the fact that $s<1$ to use first order differences in $\norm{f}_{F^s_{p,q}(\R^d)}^p$.
\end{proof}

\begin{definition}\label{defFspq}
Let $X(\R^d)$ be a Banach space of measurable functions in $\R^d$. Let $U\subset \R^d$ be a open set. Then for every measurable function $f:U\to \C$ we define
$$\norm{f}_{X(U)} := \inf_{g\in X(\R^d):\, g|_U\equiv f} \norm{g}_{X(\R^d)}.$$  
\end{definition}

Next we introduce a norm which will be the main tool for the proofs in this paper.
\begin{definition}\label{defAspqU}
Consider $1\leq p<\infty$, $1\leq q \leq \infty$ and $0<s<1$ with $s> \frac{d}{p}-\frac{d}{q}$. Let $U$ be an open set in $\R^d$. We say that a locally integrable function $f\in A^s_{p,q}(U)$ if
\begin{itemize}
\item The function $f\in L^p(U)$, and
\item the seminorm
\begin{equation}\label{eqSeminormAspq}
\norm{f}_{\dot{A}^s_{p,q}(U)}:=\left(\int_U \left(\int_{U}\frac{|f(x)-f(y)|^q}{|x-y|^{sq+d}} \,dy\right)^{\frac{p}{q}}dx\right)^{\frac{1}{p}}
\end{equation}
is finite.
\end{itemize}
We define the norm 
\begin{equation*}
\norm{f}_{{A}^s_{p,q}(U)}:=\norm{f}_{L^p(U)}+\norm{f}_{\dot{A}^s_{p,q}(U)}.
\end{equation*}
\end{definition}

In some situations, the classical Besov spaces $B^s_{p,p}(U)=A^s_{p,p}(U)$ and the fractional Sobolev spaces $W^{s,p}(U)=A^s_{p,2}(U)$. For instance, when $\Omega$ is a Lipschitz domain  then $A^s_{p,2}(\Omega)=W^{s,p}(\Omega)$ (see \cite{Strichartz}). We will see that this is a property of all uniform domains. 

\begin{remark}\label{remsdpdq}
The condition $s>\frac{d}{p}-\frac{d}{q}$ ensures that the $C^\infty_c$-functions are in  the class $A^s_{p,q}(\R^d)$.
\end{remark}
\begin{proof}
Indeed, given a bump function $\varphi\in C^\infty_c (\DDD)$,
\begin{align*}
\norm{\varphi}_{A^s_{p,q}(\R^d)}
	& \geq \left(\int_{(2\DDD)^c} \left(\int_\DDD \frac{|\varphi(x)-\varphi(y)|^q}{|x-y|^{sq+d}} \,dy\right)^{\frac{p}{q}}dx\right)^{\frac{1}{p}}\\
	& \approx \left(\int_{(2\DDD)^c} \left(\int_{\DDD}| \varphi(y)|^q \,dy\right)^{\frac{p}{q}}\frac{1}{|x|^{sp+\frac{dp}{q}}} \,dx\right)^{\frac{1}{p}}
\end{align*} 
which is finite if and only if $\frac{d}{p}<s+\frac{d}{q}$. The converse implication is an exercise.
\end{proof}

Consider a given $\varepsilon$-uniform domain $\Omega$. In \cite{Jones} Peter Jones defines an extension operator $\Lambda_0:W^{1,p}(\Omega) \to W^{1,p}(\R^d)$ for  $1<p<\infty$, that is, a bounded operator such that $\Lambda_0 f|_\Omega\equiv f|_\Omega$ for every $f\in W^{1,p}(\Omega)$. This extension operator is used to prove that the intrinsic characterization of $W^{1,p}(\Omega)$ given by
$$\norm{f}_{W^{1,p}(\Omega)}\approx \norm{f}_{L^p(\Omega)}+\norm{\nabla f}_{L^p(\Omega)}$$
is equivalent to the restriction norm.

Next we will see that the same operator is an extension operator for $A^s_{p,q}(\Omega)$ for $0<s<1$ with $s>\frac{d}{p}-\frac{d}{q}$. To define it we need a Whitney covering $\mathcal{W}_1$ of $\Omega$  (see Definition \ref{defWhitney}), a Whitney covering $\mathcal{W}_2$ of $\Omega^c$ and we define $\mathcal{W}_3$ to be the collection of cubes in $\mathcal{W}_2$ with side-lengths small enough, so that for any $Q\in \mathcal{W}_3$ there is a $S\in \mathcal{W}_1$ with $\Dist(Q,S)\leq C \ell(Q)$ and $\ell(Q)=\ell(S)$ (see \cite[Lemma 2.4]{Jones}). We define the symmetrized cube $Q^*$ as one of the cubes satisfying these properties. Note that the number of possible choices for $Q^*$ is uniformly bounded and, if $\Omega$ is an unbounded uniform domain, then 
\begin{equation}\label{eqUnboundedDomainWhitney}
\mathcal{W}_2=\mathcal{W}_3.
\end{equation}

\begin{lemma}\label{lemSymmetrized}[see \cite{Jones}]
For cubes $Q_1,Q_2\in\mathcal{W}_3$ and $S\in\mathcal{W}_1$ we have that
\begin{itemize}
\item The symmetrized cubes have finite overlapping: there exists a constant $C$ depending on the parameter $\varepsilon$ and the dimension $d$ such that $\#\{Q\in\mathcal{W}_3: Q^*=S\}\leq C$.
\item The long distance is invariant in the following sense:
\begin{equation}\label{eqLongDistanceInvariant}
\Dist(Q_1^*,Q_2^*)\approx \Dist(Q_1,Q_2) \mbox{\quad\quad and \quad\quad}\Dist(Q_1^*,S)\approx \Dist(Q_1,S) 
\end{equation}
\item In particular, if ${Q_1}\cap2{Q_2}\neq \emptyset$ ($Q_1$ and $Q_2$ are neighbors by \rf{eqWhitney5}), then $\Dist(Q_1^*,Q_2^*)\approx \ell(Q_1)$.
\end{itemize}
\end{lemma}

We define the family of bump functions $\{\psi_Q\}_{Q\in \mathcal{W}_2}$ to be a partition of the unity associated to $\left\{\frac{11}{10}Q\right\}_{Q\in\mathcal{W}_2}$, that is, their sum $\sum\psi_Q\equiv 1$, they satisfy the pointwise inequalities $0\leq \psi_Q\leq \chi_{\frac{11}{10}Q}$ and $\norm{\nabla\psi_Q}_\infty \lesssim \frac{1}{\ell(Q)}$.  We can define the operator
$$\Lambda_0 f(x)= f(x)\chi_\Omega(x) + \sum_{Q\in\mathcal{W}_3} \psi_Q(x) f_{Q^*} \mbox{ for any }f\in L^1_{loc}(\Omega)$$
 (recall that $f_{U}$ stands for the mean of a function $f$ in a set $U$). This function is defined almost everywhere because the boundary of the domain $\Omega$ has zero Lebesgue measure (see \cite[Lemma 2.3]{Jones}).

\begin{lemma}\label{lemExtensionOperator}
Let $\Omega$ be a  uniform domain, let $1<p,q<\infty$ and $0<s<1$ with $s>\frac{d}{p}-\frac{d}{q}$. Then, $\Lambda_0: A^s_{p,q}(\Omega)\to F^s_{p,q}(\R^d)$ is an extension operator. Furthermore, $\Lambda_0f \in L^{\max\{p,q\}}$ for every $f\in A^s_{p,q}(\Omega)$.
\end{lemma}

\begin{proof}
We have to check that  
\begin{equation*}
\norm{\Lambda_0 f}_{{A}^s_{p,q}(\R^d)}=\norm{\Lambda_0f}_{L^p}+\left(\int_{\R^d} \left(\int_{\R^d} \frac{| \Lambda_0 f(x)-\Lambda_0 f(y)|^q}{|x-y|^{sq+d}} \,dy\right)^{\frac{p}{q}}dx\right)^{\frac{1}{p}}\lesssim \norm{f}_{A^s_{p,q}(\Omega)}.
\end{equation*}

First, note that $\norm{\Lambda_0f}_{L^p}\leq \norm{f}_{L^p(\Omega)}+\norm{\Lambda_0f}_{L^p(\Omega^c)}$. By Jensen's inequality, we have that 
$$\norm{\Lambda_0f}_{L^p(\Omega^c)}^p\lesssim_p \sum_{Q\in\mathcal{W}_3} |f_{Q^*}|^p\norm{\psi_Q}_{L^p}^p \leq \sum_{Q\in\mathcal{W}_3} \frac{1}{\ell(Q)^d}\norm{f}_{L^p(Q^*)}^p  \left(\frac{11}{10}\ell(Q)\right)^d . $$
By the finite overlapping of the symmetrized cubes, 
$$\norm{\Lambda_0f}_{L^p(\Omega^c)}^p\lesssim \norm{f}_{L^p(\Omega)}^p. $$
The same can be said about $L^q$ when $q>p$. In that case, moreover, one can cover $\Omega$ with balls $\{B_j\}_{j\in J}$ with radius one such that $|B_j\cap\Omega|\approx 1$. Then, using the subadditivity of  $x\mapsto |x|^{\frac{p}{q}}$ we get
\begin{align}\label{eqPimpliesQ}
\norm{f}_{L^q(\Omega)}^p
			& \leq \left(\sum_j \int_{B_j\cap \Omega}|f(y)|^q \, dy\right)^\frac{p}{q} \\
\nonumber	& \lesssim_q \sum_j \left(\fint_{B_j\cap \Omega} \left(\int_{B_j\cap \Omega}|f(y)-f(x)|^q \, dy\right)^\frac{p}{q} dx + \fint_{B_j\cap \Omega}  \left(\int_{B_j\cap \Omega}|f(x)|^q \, dy\right)^\frac{p}{q}dx\right)\\ 
\nonumber	& \lesssim  \int_{\Omega} \left(\int_{\Omega} \frac{|f(x)-f(y)|^q}{|x-y|^{sq+d}} \, dy\right)^\frac{p}{q} dx + \norm{f}_{L^p(\Omega)}^p \approx \norm{f}_{A^s_{p,q}(\Omega)}^p,
\end{align}
by Definition \ref{defAspqU}.

 It remains to check that
\begin{equation*}
\norm{\Lambda_0 f}_{{\dot A}^s_{p,q}(\R^d)}=\left(\int_{\R^d} \left(\int_{\R^d} \frac{| \Lambda_0 f(x)-\Lambda_0 f(y)|^q}{|x-y|^{sq+d}} \,dy\right)^{\frac{p}{q}}dx\right)^{\frac{1}{p}}\lesssim \norm{f}_{A^s_{p,q}(\Omega)}.
\end{equation*}
More precisely, we will prove that 
\begin{align*}
\circled{a}+\circled{b}+\circled{c} \lesssim \norm{f}_{A^s_{p,q}(\Omega)}^p,
\end{align*}
where
\begin{align*}
\circled{a} 
	& := \int_\Omega \left(\int_{\Omega^c} \frac{| f(x)-\Lambda_0 f(y)|^q}{|x-y|^{sq+d}} \,dy\right)^\frac{p}{q}dx,	
& \circled{b} 
	& := \int_{\Omega^c} \left(\int_{\Omega} \frac{| \Lambda_0 f(x)- f(y)|^q}{|x-y|^{sq+d}} \,dy\right)^\frac{p}{q}dx\mbox{\quad and }\\
\circled{c}
	& :=\int_{\Omega^c}\left(\int_{\Omega^c} \frac{| \Lambda_0 f(x)-\Lambda_0 f(y)|^q}{|x-y|^{sq+d}} \,dy\right)^\frac{p}{q}dx.
\end{align*}

Let us begin with 
\begin{align*}
\circled{a}
	& = \int_\Omega \left(\int_{\Omega^c} \frac{| f(x)-\sum_{S\in\mathcal{W}_3} \psi_{S}(y) f_{S^*}|^q}{|x-y|^{sq+d}} \,dy\right)^\frac{p}{q}dx.
\end{align*}
Call $\mathcal{W}_4:=\{S\in\mathcal{W}_3: \mbox{ all the neighbors of $S$ are in } \mathcal{W}_3\}$. Given $y\in \frac{11}{10}S$, where $S\in \mathcal{W}_4$, we have that $\sum_{P\in\mathcal{W}_3} \psi_P(y)\equiv 1$ and, otherwise $0\leq 1- \sum_{P\in\mathcal{W}_3} \psi_P(y)\leq 1$. Thus
\begin{align*}
\circled{a}
	&  \lesssim \sum_{Q\in\mathcal{W}_1}\int_Q \left(\sum_{S\in\mathcal{W}_3} \frac{\left| f(x)-f_{S^*}\right|^q}{\Dist(Q,S)^{sq+d}} \int_{\frac{11}{10}S} \psi_{S}(y) \,dy \right)^\frac{p}{q}dx \\
	& \quad + \sum_{Q\in\mathcal{W}_1}\int_Q \left(\sum_{S\in\mathcal{W}_2\setminus\mathcal{W}_4}\int_{\frac{11}{10}S}\frac{\left| \left(1- \sum_{P\in\mathcal{W}_3} \psi_P(y)\right)f(x)\right|^q}{\Dist(Q,S)^{sq+d}}\,dy\right)^\frac{p}{q}dx =:\circled{a1}+\circled{a2}.
\end{align*}
In $\circled{a1}$ by the choice of the symmetrized cube we have that $\int_{\frac{11}{10}S}  \psi_{S}(y)  \,dy \approx \ell(S^*)^d$. Jensen's inequality implies that $\left| f(x)-f_{S^*}\right|^q\leq \frac{1}{\ell(S^*)^d}\int_{S^*}|f(x)-f(\xi)|^q\, d\xi$. By \rf{eqLongDistanceInvariant} and the finite overlapping of the symmetrized cubes, we get that 
\begin{align*}
\circled{a1}
	& \lesssim \sum_{Q\in\mathcal{W}_1}\int_Q \left(\sum_{S\in\mathcal{W}_3}\int_{S^*} \frac{\left| f(x)-f(\xi)\right|^q}{\Dist(Q,S^*)^{sq+d}} \,d\xi \right)^\frac{p}{q}dx \lesssim \norm{f}_{\dot A^s_{p,q}(\Omega)}^p.
\end{align*}
	
To bound $\circled{a2}$ just note that for $Q\in\mathcal{W}_1$ and $S\in\mathcal{W}_2\setminus\mathcal{W}_4$, we have that $S$ is far from the boundary, say $\ell(S)\geq \ell_0$, where $\ell_0$ depends only on $\diam(\Omega)$ and $\varepsilon$ and, if $\Omega$ is unbounded, then $\ell_0=\infty$ and $\circled{a2}=0$ by \rf{eqUnboundedDomainWhitney}. Thus, we have that
$$\circled{a2}\lesssim \sum_{Q\in\mathcal{W}_1}\int_Q \left(\sum_{S\in\mathcal{W}_2\setminus\mathcal{W}_4}\int_{\frac{11}{10}S}\frac{\left| f(x)\right|^q}{\Dist(Q,S)^{sq+d}}\,dy\right)^\frac{p}{q}dx \lesssim \left(\sum_{S\in\mathcal{W}_2\setminus\mathcal{W}_4} \frac{\ell(S)^d}{\Dist(\Omega,S)^{sq+d}}\right)^\frac{p}{q}\norm{f}_{L^p}^p.$$
Recall that Whitney cubes have side-length equivalent to their distance to $\partial\Omega$. Moreover, the number of cubes of a given side-length bigger than $\ell_0$ is uniformly bounded when $\Omega$ is bounded, so $\sum_{S\in\mathcal{W}_2\setminus\mathcal{W}_4} \frac{\ell(S)^d}{\ell(S)^{sq+d}}$ is a geometric sum. Therefore,
$$\circled{a2}\lesssim \left(\sum_{S\in\mathcal{W}_2\setminus\mathcal{W}_4} \frac{1}{\ell(S)^{sq}}\right)^\frac{p}{q}\norm{f}_{L^p}^p \leq C_{\varepsilon, \diam(\Omega)}\ell_0^{-sp} \norm{f}_{L^p}^p.$$

Next, note that, using the same decomposition as above, we have that
\begin{align*}
\circled{b}
	& =\int_{\Omega^c} \left(\int_\Omega  \frac{|\sum_{Q\in\mathcal{W}_3} \psi_{Q}(x) f_{Q^*}- f(y)|^q}{|x-y|^{sq+d}}\,dy\right)^\frac{p}{q}dx \\
	& \lesssim \sum_{Q\in\mathcal{W}_3}\int_{\frac{11}{10}Q} \psi_{Q}(x)^p \,dx\left( \sum_{S\in\mathcal{W}_1}\int_S \frac{\left|f_{Q^*}- f(y)\right|^q}{\Dist(Q,S)^{sq+d}} \,dy \right)^\frac{p}{q} \\
	& \quad + \sum_{P\in\mathcal{W}_2\setminus\mathcal{W}_4}\int_{P}\left(1- \sum_{Q\in \mathcal{W}_3}\psi_{Q}(x) \right)^p \,dx \left( \sum_{S\in\mathcal{W}_1}\int_S \frac{\left| f(y)\right|^q}{\Dist(P,S)^{sq+d}} \,dy \right)^\frac{p}{q} = :\circled{b1}+\circled{b2}.
\end{align*}
We have that
\begin{align*}
\circled{b1}
	& \lesssim \sum_{Q\in\mathcal{W}_3} \ell(Q)^d\left(\sum_{S\in\mathcal{W}_1}\int_S \frac{\left(\frac{1}{\ell(Q)^{d}} \int_{Q^*}\left| f(\xi)-f(y)\right| \,d\xi \right)^q }{\Dist(Q^*,S)^{sq+d}} \, dy \right)^\frac{p}{q} 
\end{align*}
and, thus, by Minkowsky's integral inequality (see \cite[Appendix A1]{SteinPetit}), we have that
\begin{align*}
\circled{b1}
	& \lesssim \sum_{Q\in\mathcal{W}_3} \frac{\ell(Q)^d}{\ell(Q)^{dp}}  \left(\int_{Q^*}\left(\sum_{S\in\mathcal{W}_1}\int_S\frac{\left| f(\xi)-f(y)\right|^q}{|\xi-y|^{sq+d}} \,dy  \right)^\frac{1}{q}  \,d\xi \right)^p .
\end{align*}
By H\"older's inequality and the finite overlapping of symmetrized cubes, we get that
\begin{align*}
\circled{b1}
	& \lesssim \sum_{Q\in\mathcal{W}_3} \frac{1}{\ell(Q)^{d(p-1)}} \int_{Q^*}\left( \int_\Omega\frac{\left| f(\xi)-f(y)\right|^q}{|\xi-y|^{sq+d}} \,dy  \right)^\frac{p}{q}  \,d\xi  \ell(Q)^\frac{dp}{p'}	\lesssim \int_{\Omega}\left( \int_\Omega\frac{\left| f(\xi)-f(y)\right|^q}{|\xi-y|^{sq+d}} \,dy  \right)^\frac{p}{q}  \,d\xi ,
\end{align*}
that is,
$$\circled{b1} \lesssim \norm{f}_{\dot A^s_{p,q}(\Omega)}^p.$$
To bound $\circled{b2}$, note that as before, if $\Omega$ is unbounded, then $\circled{b2}=0$ and, otherwise, we have that
$$\circled{b2}\approx \sum_{Q\in\mathcal{W}_2\setminus\mathcal{W}_4}\ell(Q)^d \left( \sum_{S\in\mathcal{W}_1}\int_{S} \frac{\left| f(y)\right|^q}{\Dist(Q,\Omega)^{sq+d}}  \,dy \right)^\frac{p}{q} \lesssim \norm{f}_{L^q(\Omega)}^p \sum_{Q\in\mathcal{W}_2\setminus\mathcal{W}_4}\frac{\ell(Q)^d}{\dist(Q,\Omega)^{sp+\frac{dp}{q}}} .$$
Now, since $s>\frac{d}{p}-\frac{d}{q}$ we have that $sp+\frac{dp}{q}>d$. Therefore,
$$\sum_{Q\in\mathcal{W}_2\setminus \mathcal{W}_4}\frac{\ell(Q)^d}{\dist(Q,\Omega)^{sp+\frac{dp}{q}}}\approx\sum_{Q\in\mathcal{W}_2\setminus \mathcal{W}_4}\frac{1}{\ell(Q)^{sp+\frac{dp}{q}-d}}\leq C_{\varepsilon, \diam(\Omega)}\ell_0^{d-sp-\frac{dp}{q}}.$$
 On the other hand, if $\Omega$ is bounded and $q\leq p $, then $\norm{f}_{L^q(\Omega)}\lesssim \norm{f}_{L^p(\Omega)}$ by the H\"older inequality and, if $p<q$, then $\norm{f}_{L^q(\Omega)}\lesssim \norm{f}_{A^s_{p,q}(\Omega)}$ by \rf{eqPimpliesQ}.

Let us focus on $\circled{c}$.  We have that
\begin{align*}
\circled{c}
	& = \int_{\Omega^c} \left(\int_{\Omega^c}  \frac{|\sum_{P\in\mathcal{W}_3} \psi_{P}(x) f_{P^*}-\sum_{S\in\mathcal{W}_3} \psi_{S}(y) f_{S^*}|^q}{|x-y|^{sq+d}}\,dy\right)^\frac{p}{q}dx .
\end{align*}
Given $x\in \frac{11}{10}Q$ where $Q\in \mathcal{W}_4$ and $y\in \Omega^c\cap B(x, \frac{\ell_0}{10})$, then neither $x$ nor $y$ are in the support of any bump function of a cube in $\mathcal{W}_2\setminus\mathcal{W}_3$, so $\sum_{P\in\mathcal{W}_3} \psi_P(y)\equiv 1$ and $\sum_{P\in\mathcal{W}_3} \psi_P(x)\equiv 1$. Therefore
$$\sum_{P\in\mathcal{W}_3} \psi_{P}(x) f_{P^*}-\sum_{S\in\mathcal{W}_3} \psi_{S}(y) f_{S^*}=  \sum_{P\cap 2Q\neq \emptyset}  \sum_{S\in\mathcal{W}_3} \psi_{P}(x)\psi_{S}(y)\left(f_{P^*}-f_{S^*}\right).$$
If, moreover, $y\in B\left(x,\frac{1}{10}\ell(Q)\right)$, since the points are `close' to each other, we will use the H\"older regularity of the bump functions, so we write 
$$\sum_{P\in\mathcal{W}_3} \psi_{P}(x) f_{P^*}-\sum_{S\in\mathcal{W}_3} \psi_{S}(y) f_{S^*}= \sum_{P\in\mathcal{W}_3}\left( \psi_{P}(x)-\psi_{P}(y)\right)f_{P^*}.$$
This decomposition is still valid if $Q\in \mathcal{W}_2\setminus \mathcal{W}_4$ and $y\in B\left(x,\frac{1}{10}\ell(Q)\right)$, that is,  $y\in B\left(x,\frac{\ell_0}{10}\right)$, but we will treat this case apart since we lose the cancellation of the sums of bump functions but we gain a uniform lower bound on the side-lengths of the cubes involved.
Finally, we will group the remaining cases, when $x\in\Omega^c$ and $y\notin B(x, \frac{\ell_0}{10})$ in an error term.
Considering all these facts we get
\begin{align*}	
\circled{c}	
	& \lesssim \sum_{Q\in\mathcal{W}_4}\int_{Q} \left(\int_{\Omega^c \setminus B\left(x,\frac{1}{10}\ell(Q)\right)} \sum_{P\cap 2Q\neq \emptyset}  \sum_{S\in\mathcal{W}_3}\left|\psi_{P}(x) \psi_{S}(y)\right|\frac{\left|f_{P^*}- f_{S^*}\right|^q}{\Dist(P^*,S^*)^{sq+d}} \,dy\right)^\frac{p}{q}dx\\
	& \quad +  \sum_{Q\in\mathcal{W}_4}\int_{Q} \left(  \int_{B\left(x,\frac{1}{10}\ell(Q)\right)} \frac{| \sum_{S\cap 2Q\neq \emptyset} \left(\psi_{S}(x)- \psi_{S}(y) \right)f_{S^*}|^q}{|x-y|^{sq+d}} \,dy\right)^\frac{p}{q}dx\\
	& \quad + \sum_{Q\in\mathcal{W}_2\setminus\mathcal{W}_4}\int_{Q} \left(\int_{B\left(x,\frac{\ell_0}{10}\right)} \frac{| \sum_{S\in \mathcal{W}_3: S\cap 2Q\neq \emptyset} \left(\psi_{S}(x)- \psi_{S}(y) \right)f_{S^*}|^q}{|x-y|^{sq+d}} \,dy\right)^\frac{p}{q}dx\\
	& \quad + \int_{\Omega^c}\left(\int_{\Omega^c\setminus B\left(x,\frac{\ell_0}{10}\right)} \frac{| \Lambda_0 f(x)-\Lambda_0 f(y)|^q}{|x-y|^{sq+d}} \,dy\right)^\frac{p}{q}dx\\
	&  =:\circled{c1}+\circled{c2}+\circled{c3}+\circled{c4},
\end{align*}
where the last two terms vanish in case $\Omega$ is bounded. 

Using the same arguments as in $\circled{a1}$ and $\circled{b1}$ we have that
$$\circled{c1}\lesssim \norm{f}_{\dot A^s_{p,q}(\Omega)}^p.$$
Also combining the arguments used to bound $\circled{a2}$ and $\circled{b2}$ we get that if $\Omega$ is bounded, then
$$\circled{c4}\lesssim \left(\norm{f}_{L^p(\Omega)}+\norm{f}_{L^q(\Omega)}\right)^p,$$
and it vanishes otherwise.

The novelty comes from the fact that we are integrating in $\Omega^c$ both terms in $\circled{c}$, so the variables in the integrals $\circled{c2}$ and $\circled{c3}$ can get as close as one can imagine. Here we need to use the smoothness of the bump functions, but also the smoothness of $f$ itself. The trick for $\circled{c2}$ is to use that $\{\psi_Q\}$ is a partition of the unity with $\psi_Q$ supported in $\frac{11}{10}Q$, that is, $\sum_{S\in\mathcal{W}_3}\psi_S(x)=\sum_{S\cap 2Q\neq \emptyset}\psi_S(x)=1$ if $x \in \frac{11}{10}Q$ with $Q \in \mathcal{W}_4$. Thus,
\begin{align*}
\circled{c2}
	& =   \sum_{Q\in\mathcal{W}_4}\int_{Q} \left( \int_{B\left(x,\frac{1}{10}\ell(Q)\right)} \frac{| \sum_{S\cap 2Q\neq \emptyset} \left(\psi_{S}(x)- \psi_{S}(y) \right)\left(f_{S^*}-f_{Q^*}\right)|^q}{|x-y|^{sq+d}} \,dy\right)^\frac{p}{q}dx,
\end{align*}
and using the fact that $\norm{\nabla\psi_Q }_\infty\lesssim\frac{1}{\ell(Q)}$ and \rf{eqMaximalClose}, we have that
\begin{align*}
\circled{c2}
	& \lesssim_{q}    \sum_{Q\in\mathcal{W}_4}\int_{Q} \left(\sum_{S\cap 2Q\neq \emptyset} \left| f_{S^*}-f_{Q^*}\right|^q \int_{B\left(x,\frac{1}{10}\ell(Q)\right)} \frac{\left|x-y \right|^q }{\ell(Q)^q}\frac{1}{|x-y|^{sq+d}} \,dy\right)^\frac{p}{q}dx\\
	& \lesssim_{s}    \sum_{Q\in\mathcal{W}_4} \ell(Q)^d \left(\frac{ \sum_{S\cap 2Q\neq \emptyset} \left| f_{S^*}-f_{Q^*}\right|^q}{\ell(Q)^{sq}} \right)^\frac{p}{q} \approx  \sum_{Q\in\mathcal{W}_4} \ell(Q)^d \left( \sum_{S\cap 2Q\neq \emptyset} \frac{\left| f_{S^*}-f_{Q^*}\right|^q}{\Dist(Q^*,S^*)^{sq}} \right)^\frac{p}{q},
\end{align*}
which can be bounded as $\circled{c1}$.

Finally, we bound the error term $\circled{c3}$, assuming $\Omega$ to be a bounded domain. Here we cannot use the cancellation of the partition of the unity anymore. Instead, we will use the $L^p$ norm of $f$, the H\"older regularity of the bump functions and the fact that all the cubes considered are roughly of the same size:
\begin{align*}
\circled{c3}
	& =   \sum_{Q\in\mathcal{W}_2\setminus \mathcal{W}_4}\int_{Q} \left( \int_{B\left(x,\frac{\ell_0}{10}\right)} \frac{| \sum_{S\cap 2Q\neq \emptyset} \left(\psi_{S}(x)- \psi_{S}(y) \right)f_{S^*}|^q}{|x-y|^{sq+d}} \,dy\right)^\frac{p}{q}dx \\
	& \lesssim  \sum_{\substack{Q\in\mathcal{W}_2\\ \ell_0\leq \ell(Q)\leq 2\ell_0}} \int_{Q} \sum_{\substack{S\in\mathcal{W}_3 \\ S\cap 2Q\neq \emptyset}} \left|f_{S^*}\right|^p\left( \int_{B\left(x,\frac{\ell_0}{10}\right)}\frac{1}{\ell_0^{q}} \frac{1}{|x-y|^{(s-1)q+d}} \,dy\right)^\frac{p}{q}dx\\
	& \lesssim_{\varepsilon,\ell_0,q,p} \sum_{\substack{S\in\mathcal{W}_3\\ \frac{\ell_0}{2}\leq\ell(S)\leq\ell_0}} \norm{f}_{L^p(S^*)}^p\lesssim \norm{f}_{L^p(\Omega)}^p.
\end{align*}
\end{proof}

\begin{corollary}\label{coroExtensionDomain}
Let $\Omega$ be a  uniform domain with an admissible Whitney covering $\mathcal{W}$. Given $1<p<\infty$, $1<q<\infty$ and $0<s<1$ with $s>\frac{d}{p}-\frac{d}{q}$,  we have that
$A^s_{p,q}(\Omega)=F^s_{p,q}(\Omega)$, and
$$\norm{f}_{F^s_{p,q}(\Omega)}\approx \norm{f}_{A^s_{p,q}(\Omega)} \mbox{\quad\quad for all } f\in F^s_{p,q}(\Omega).$$
\end{corollary}

\begin{proof}
By Corollary  \ref{coroCharacterizationRd}, given $f\in F^s_{p,q}(\Omega)$ we have that
$$\norm{f}_{A^s_{p,q}(\Omega)}\leq \inf_{g\in L^{\max\{p,q\}}: g|_\Omega\equiv f} \norm{g}_{A^s_{p,q}(\R^d)} \approx\inf_{g: g|_\Omega\equiv f} \norm{g}_{F^s_{p,q}(\R^d)}=\norm{f}_{F^s_{p,q}(\Omega)}.$$
By the Lemma \ref{lemExtensionOperator} we have the converse. Given $f\in A^s_{p,q}(\Omega)$ we have that
$$\norm{f}_{F^s_{p,q}(\Omega)} =\inf_{g : g|_\Omega\equiv f} \norm{g}_{F^s_{p,q}(\R^d)} \leq \norm{\Lambda_0 f}_{F^s_{p,q}(\R^d)}\approx \norm{\Lambda_0 f}_{A^s_{p,q}(\R^d)} \leq C \norm{f}_{A^s_{p,q}(\Omega)}.$$
\end{proof}

\section{Equivalent norms with reduction of the integration domain.}\label{secNorms}
Next we present an equivalent  norm  for $F^s_{p,q}(\Omega)$ in terms of differences but reducing the domain of integration of the inner variable to the shadow of the outer variable in the seminorm $\norm{\cdot}_{\dot{A}^s_{p,q}(\Omega)}$ defined in \rf{eqSeminormAspq}.

\begin{lemma}\label{lemFirstReduction}
Let $\Omega$ be a uniform domain with an admissible Whitney covering $\mathcal{W}$, let $1<p,q<\infty$ and $0<s<1$ with $s>\frac{d}{p}-\frac{d}{q}$. Then, $f\in F^s_{p,q}(\Omega)$ if and only if 
\begin{equation}\label{eqEquivalentNorm}
\norm{f}_{\widetilde{A}^s_{p,q}(\Omega)}= \norm{f}_{L^p(\Omega)} + \left(\sum_{Q\in\mathcal{W}} \int_Q \left(\int_{\Sh(Q)}\frac{|f(x)-f(y)|^q}{|x-y|^{sq+d}} \,dy\right)^{\frac{p}{q}}dx\right)^\frac{1}{p} <\infty.
\end{equation}
This quantity defines a norm which is equivalent to $\norm{f}_{F^s_{p,q}(\Omega)}^p$ and, moreover, we have that $f\in L^{\max\{p,q\}}(\Omega)$.
\end{lemma}

\begin{proof}
Let $\Omega$ be an $\varepsilon$-uniform domain. Recall that in  \rf{eqSeminormAspq} we defined
$$\norm{f}_{\dot{A}^s_{p,q}(\Omega)}=\left(\int_\Omega \left(\int_{\Omega}\frac{|f(x)-f(y)|^q}{|x-y|^{sq+d}} \,dy\right)^{\frac{p}{q}}dx\right)^{\frac{1}{p}}.
$$
Trivially
$$\norm{f}_{\dot A^s_{p,q}(\Omega)}^p\gtrsim \sum_{Q\in\mathcal{W}} \int_Q \left(\int_{\Sh(Q)}\frac{|f(x)-f(y)|^q}{|x-y|^{sq+d}} \,dy\right)^{\frac{p}{q}}dx.$$

To prove the converse inequality, we will use the seminorm in the duality form
\begin{equation}\label{eqSeminormDuality}
\norm{f}_{\dot{A}^s_{p,q}(\Omega)}=\sup_{\norm{g}_{L^{p'}(L^{q'}(\Omega))}\leq 1} \int_\Omega \int_{\Omega}\frac{|f(x)-f(y)|}{|x-y|^{s+\frac{d}{q}}} g(x,y) \,dy \,dx.
\end{equation}
Let $g>0$ be an $L^1_{loc}$ function with $\norm{g}_{L^{p'}(L^{q'}(\Omega))}\leq 1$.
 Since  the shadow of every cube $Q$ contains $2Q$, we just use H\"older's inequality to find that
 \begin{equation}\label{eqIntegralStrichartzClose}
\sum_{Q\in\mathcal{W}} \int_Q\int_{ 2Q} \frac{|f(x)-f(y)|}{|x-y|^{s+\frac{d}{q}}} g(x,y) \,dy\, dx \leq\left(\sum_{Q\in\mathcal{W}} \int_Q \left(\int_{2Q} \frac{|f(x)-f(y)|^q}{|x-y|^{sq+d}} \,dy\right)^\frac{p}{q} dx\right)^\frac{1}{p}.
\end{equation}
Therefore, we only need to prove the estimate
\begin{equation}\label{eqStrichartzDuality}
\sum_{Q,S} \int_Q\int_{S\setminus 2Q} \frac{|f(x)-f(y)|}{|x-y|^{s+\frac{d}{q}}} g(x,y) \,dy\, dx
\lesssim  \left(\sum_{Q\in\mathcal{W}} \int_Q \left(\int_{\Sh(Q)}\frac{|f(x)-f(y)|^q}{|x-y|^{sq+d}}\, dy\right)^{\frac{p}{q}}dx\right)^\frac{1}{p}.
\end{equation}

If $x\in Q$, $y\in S\setminus 2Q$, then $|x-y|\approx \Dist(Q,S)$, so we can write
\begin{align}\label{eqStrichartzFarNear}
\sum_{Q,S} \int_Q \int_{S\setminus 2Q} \frac{|f(x)-f(y)|}{|x-y|^{s+\frac{d}{q}}} g(x,y) \,dy \,dx
	& \lesssim  \sum_{Q,S} \int_Q\int_{S}\frac{|f(x)-f(y)|}{\Dist(Q,S)^{s+\frac{d}{q}}} g(x,y) \,dy \,dx.
\end{align}
Since $\Omega$ is a uniform domain, for every pair of cubes $Q$ and $S$ in this sum, there exists an admissible chain $[Q,S]$ joining them. Thus, writing $f_Q=\fint_Q f\, dm$ for the mean of $f$ in $Q$, the right-hand side of \rf{eqStrichartzFarNear} can be split as follows:
\begin{align}\label{eqStrichartzInitial}
 \sum_{Q,S}\int_Q\int_{S} \frac{|f(x)-f(y)|}{\Dist(Q,S)^{s+\frac{d}{q}}} g(x,y) \,dy \,dx
\nonumber	& \leq \sum_{Q,S}\int_Q\int_{S}\frac{|f(x)-f_Q|}{\Dist(Q,S)^{s+\frac{d}{q}}} g(x,y) \,dy \,dx\\
\nonumber	& \quad+ \sum_{Q,S}\int_Q\int_{S}\frac{|f_{Q}-f_{Q_S}| }{\Dist(Q,S)^{s+\frac{d}{q}}} g(x,y) \,dy \,dx\\
\nonumber	& \quad+ \sum_{Q,S}\int_Q\int_{S}\frac{|f_{Q_S}-f(y)| }{\Dist(Q,S)^{s+\frac{d}{q}}} g(x,y) \,dy \,dx\\
			& =: \circled{1}+\circled{2}+\circled{3}.
\end{align}
Note that the definition of $Q_S$ depends on the chosen chain.

The first term can be immediately bounded by the Cauchy-Schwarz inequality. Namely, writing $G(x)=\norm{g(x,\cdot)}_{L^{q'}(\Omega)}$, by \rf{eqMaximalAllOver}  we have that
\begin{align*}
\circled{1}
	& \leq \sum_{Q\in\mathcal{W}} \int_Q |f(x)-f_Q| \left(\sum_{S\in\mathcal{W}}\int_S g(x,y)^{q'} \, dy\right)^{\frac1{q'}} \left(\sum_{S\in\mathcal{W}}\frac{\ell(S)^d}{\Dist(Q,S)^{sq+d}}\right)^{\frac1q}  \,dx \\
	& \leq \sum_{Q\in\mathcal{W}} \frac{\int_Q |f(x)-f_Q| G(x)   \,dx}{\ell(Q)^{s}}.
\end{align*}
By Jensen's inequality, $|f(x)-f_Q|\leq \left( \frac{1}{\ell(Q)^d} \int_Q |f(x)-f(y)|^q\, dy\right)^{\frac{1}{q}}$ and thus, since $\ell(Q)\gtrsim_d |x-y|$ for $x,y\in Q$, we have that
\begin{align}\label{eqBound1}
\circled{1}
	& \lesssim \left(\sum_{Q\in\mathcal{W}} \int_Q \left( \int_Q \frac{|f(x)-f(y)|^q}{|x-y|^{sq+d}}\, dy \right)^{\frac{p}{q}}dx\right)^\frac{1}{p} \norm{G}_{L^{p'}}.
\end{align}
Since $\norm{G}_{L^{p'}}=\norm{g}_{L^{p'}(L^{q'})}\leq1$, this finishes this part.

For the second one, for all cubes $Q$ and $S$ we consider the subchain $[Q,Q_S)\subset[Q,S]$. Then
\begin{align*}
\circled{2}
	& \leq \sum_{\substack{Q,S}}\int_Q\int_S\frac{g(x,y)}{\Dist(Q,S)^{s+\frac{d}{q}}}  \,dy \,dx\sum_{P\in[Q,Q_S)} |f_P-f_{\mathcal{N}(P)}|.
\end{align*}
Recall that all the cubes $P\in[Q,Q_S]$ contain $Q$ in their shadow and the properties of the Whitney covering grant that $\mathcal{N}(P)\subset 5P$. Moreover, by \rf{eqNotVeryFar} we have that $\Dist(Q,S)\approx\Dist(P,S)$. Thus, \begin{align*}
\circled{2}
	& \lesssim_d \sum_P \fint_{P} \fint_{5P} |f(\xi)-f(\zeta)| \, d\zeta \,d\xi   \sum_{Q\in \SH(P)} \int_Q \sum_{S\in\mathcal{W}}\int_S\frac{g(x,y) }{\Dist(P,S)^{s+\frac{d}{q}}} \,dy\,dx
\end{align*}
and, using H\"older's inequality and  \rf{eqMaximalAllOver}, we have that
\begin{align*}
\circled{2}
	& \lesssim \sum_P \fint_{P} \fint_{5P} |f(\xi)-f(\zeta)| \, d\zeta \,d\xi   \sum_{Q\in \SH(P)} \int_Q \left( \int_\Omega g(x,y)^{q'} \,dy\right)^{\frac1{q'}}\left( \sum_{S\in\mathcal{W}} \frac{\ell(S)^d }{\Dist(P,S)^{sq+d}} \right)^{\frac1q}\,dx \\
	& \lesssim_{d,s,q} \sum_P \fint_{P} \fint_{5P} |f(\xi)-f(\zeta)| \, d\zeta \,d\xi  \sum_{Q\in \SH(P)} \int_Q G(x)\,dx  \frac{1}{\ell(P)^{s}}.
\end{align*}

By  \rf{eqMaximalClose} we have that $ \int_{\Sh(P)} G(x)\,dx  \lesssim_{d,\varepsilon} \inf_{y\in P} MG(y) \ell(P)^{d}$, so
\begin{align*}
\circled{2}
	& \lesssim \sum_P \int_{P} \int_{5P} |f(\xi)-f(\zeta)| \, d\zeta MG(\xi)\,d\xi \frac{\ell(P)^{d-s}}{\ell(P)^{2 d}} \\
	& \lesssim_{d,p} \sum_P \int_{P}\left( \int_{5P} |f(\xi)-f(\zeta)|^q\, d\zeta\right)^{\frac1q} \ell(P)^{\frac{d}{q'}} MG(\xi) \, d\xi \frac{1}{\ell(P)^{d+s}}.
\end{align*}
Note that for $\xi, \zeta\in 5P$, we have that $|\xi-\zeta|\lesssim_d \ell(P)$. Thus, using H\"older's inequality again and the fact that $\norm{MG}_{L^{p'}}\lesssim_p \norm{G}_{L^{p'}}\leq1$, we bound the second term by
\begin{align}\label{eqBound2}
\circled{2}
	 \lesssim \sum_P \int_{P}\left( \int_{5P} \frac{|f(\xi)-f(\zeta)|^q}{|\xi-\zeta|^{sq+d}}\, d\zeta\right)^{\frac1q} MG(\xi) \, d\xi
			 \lesssim \left(\sum_P \int_{P}\left( \int_{5P}\frac{|f(\xi)-f(\zeta)|^q}{|\xi-\zeta|^{sq+d}}\, d\zeta\right)^{\frac{p}{q}} d\xi\right)^{\frac1p}.
\end{align}

Now we face the boundedness of 
$$\circled{3}=\sum_{Q,S}\int_Q\int_{S}\frac{|f_{Q_S}-f(y)| }{\Dist(Q,S)^{s+\frac{d}{q}}} g(x,y) \,dy \,dx.$$
Given two cubes $Q$ and $S$, we have that for every admissible chain $[Q,S]$ the cubes $Q,S\in \SH(Q_S)$ by Definition \ref{defUniform} and $\Dist(Q,S)\approx \ell(Q_S)$ by \rf{eqAdmissible2}. Thus, we can reorder the sum, writing
\begin{align}\label{eq3Reordered}
\circled{3}
	& \lesssim \sum_R \sum_{Q\in \SH(R)} \sum_{S\in \SH(R)} \int_Q\int_S\frac{|f_{R}-f(y)|}{\ell(R)^{s+\frac{d}{q}}} g(x,y) \,dy \,dx\\
\nonumber	& \leq \sum_R \int_R \sum_{Q\in \SH(R)} \sum_{S\in \SH(R)} \int_Q\int_S  \frac{|f(\xi)-f(y)|}{\ell(R)^{s+\left(1+\frac{1}{q}\right)d}} g(x,y) \,dy \,dx \, d\xi.
\end{align}
Using H\"older's inequality, Lemma \ref{lemMaximal} and the fact that for $S\in \SH(R)$ one has $\ell(R)\approx \Dist(S,R)$, we get that
\begin{align*}
\circled{3}
	& \lesssim \sum_R \int_R \frac{1}{\ell(R)^{s+\left(1+\frac{1}{q}\right)d}}\sum_{Q\in \SH(R)} \int_Q \sum_{S\in \SH(R)} \left({\int_S |f(\xi)-f(y)|^q}\, dy \right)^{\frac1q} \left(\int_S g(x,y)^{q'}\,dy\right)^\frac1{q'} \,dx\,d\xi\\
	& \leq \sum_R \int_R \frac{1}{\ell(R)^{s+\left(1+\frac{1}{q}\right)d}}\left( {\int_{\Sh(R)} |f(\xi)-f(y)|^q}\, dy \right)^{\frac1q} \sum_{Q\in \SH(R)}\int_Q G(x) \,dx\,d\xi\\
	& \lesssim \sum_R \int_R \left( {\int_{\Sh(R)} \frac{|f(\xi)-f(y)|^q}{\ell(R)^{sq+d}}\, dy} \right)^{\frac1q} \frac{1}{\ell(R)^{d}}MG(\xi) \ell(R)^d\,d\xi
\end{align*}
and, using the H\"older inequality again and the boundedness of the maximal operator in $L^{p'}$, we get
\begin{align}\label{eqBound3}
\circled{3}
\nonumber	& \lesssim \left(\sum_R \int_R \left( \int_{\Sh(R)} \frac{|f(\xi)-f(y)|^q}{|\xi-y|^{sq+d}}\, dy \right)^{\frac{p}{q}}d\xi\right)^\frac1p \norm{MG}_{L^{p'}}\\
	& \lesssim \left(\sum_R \int_R \left( \int_{\Sh(R)} \frac{|f(\xi)-f(y)|^q}{|\xi-y|^{sq+d}}\, dy \right)^{\frac{p}{q}}d\xi\right)^\frac1p.
\end{align}
Thus, by \rf{eqStrichartzInitial}, \rf{eqBound1}, \rf{eqBound2} and \rf{eqBound3}, we have that
$$\sum_{Q,S}\int_Q\int_{S} \frac{|f(x)-f(y)|}{\Dist(Q,S)^{s+\frac{d}{q}}} g(x,y) \,dy \,dx
\lesssim \left(\sum_R \int_R \left( \int_{\Sh(R)} \frac{|f(\xi)-f(y)|^q}{|\xi-y|^{sq+d}}\, dy \right)^{\frac{p}{q}}d\xi\right)^\frac1p .$$
 This fact, together with \rf{eqStrichartzFarNear}  proves \rf{eqStrichartzDuality} and thus, using \rf{eqSeminormDuality} and \rf{eqIntegralStrichartzClose}, we get that
$$\norm{f}_{{A}^s_{p,q}(\Omega)}\lesssim_{\varepsilon, s, p, q, d} \norm{f}_{\widetilde{A}^s_{p,q}(\Omega)}.$$

Finally, by  \rf{eqPimpliesQ} we have that $f\in L^{\max\{p,q\}}(\Omega)$.
\end{proof}

\begin{remark}
Note that we have proven that the homogeneous seminorms are equivalent, that is,
\begin{equation*}
\sum_{Q\in\mathcal{W}} \int_Q \left(\int_{\Sh(Q)}\frac{|f(x)-f(y)|^q}{|x-y|^{sq+d}} \,dy\right)^{\frac{p}{q}}dx \approx \norm{f}^p_{\dot A^s_{p,q}(\Omega)},
\end{equation*}
which improves \rf{eqEquivalentNorm}. 
\end{remark}

In some situations we can refine Lemma \ref{lemFirstReduction}.
\begin{lemma}\label{lemSecondReduction}
Let $\Omega$ be a uniform domain with an admissible Whitney covering $\mathcal{W}$, let $1<q\leq p <\infty$  and $\max\left\{\frac{d}{p}-\frac{d}{q},0\right\}<s<1$. Then, $f\in F^s_{p,q}(\Omega)$ if and only if 
\begin{equation*}
 \norm{f}_{L^p(\Omega)} + \left(\sum_{Q\in\mathcal{W}} \int_Q \left(\int_{5Q}\frac{|f(x)-f(y)|^q}{|x-y|^{sq+d}} \,dy\right)^{\frac{p}{q}}dx\right)^\frac{1}{p} <\infty.
\end{equation*}
Furthermore, this quantity defines a norm which is equivalent to $\norm{f}_{F^s_{p,q}(\Omega)}$.
\end{lemma}

\begin{proof}
Arguing as before by duality, we consider a function $g>0$ with $\norm{g}_{L^{p'}(L^{q'}(\Omega))}\leq 1$. Combining \rf{eqBound1} and \rf{eqBound2} we know that 
$$\sum_{Q,S}\int_Q\int_S\frac{|f(x)-f_{Q_S}|}{\Dist(Q,S)^{s+\frac{d}{q}}} g(x,y) \,dy \,dx\lesssim \left(\sum_{Q\in\mathcal{W}} \int_Q \left(\int_{5Q} \frac{|f(x)-f(y)|^q}{|x-y|^{sq+d}} \,dy\right)^{\frac{p}{q}}dx \right)^\frac{1}{p}$$
and, thus, we have
\begin{align}\label{eq3Improved}
 \sum_{Q,S}\int_Q\int_S \frac{|f(x)-f(y)|}{\Dist(Q,S)^{s+\frac{d}{q}}} g(x,y) \,dy \,dx
 	& \approx \left(\sum_{Q\in\mathcal{W}} \int_Q \left(\int_{5Q} \frac{|f(x)-f(y)|^q}{|x-y|^{sq+d}} \,dy\right)^{\frac{p}{q}}dx \right)^\frac{1}{p}+ \circled{3}.
\end{align}
where
\begin{align*}
\circled{3} 
	& := \sum_{Q,S}\int_Q\int_{S}\frac{|f_{Q_S}-f(y)| }{\Dist(Q,S)^{s+\frac{d}{q}}} g(x,y) \,dy \,dx \lesssim \sum_{R} \sum_{Q,S\in\SH(R)} \int_Q\int_S \frac{|f_{R}-f(y)|}{\ell(R)^{s+\frac{d}{q}}} g(x,y) \,dy \,dx
\end{align*}
by \rf{eq3Reordered}.

Using H\"older's inequality and Lemma \ref{lemMaximal}  we get that
\begin{align*}
\circled{3}
	& \lesssim \sum_{R}  \frac{1}{\ell(R)^{s+\frac{d}{q}}}\left(\sum_{S\in \SH(R)} \int_S {|f_R-f(y)|^q} \, dy \right)^{\frac{1}{q}} \sum_{Q\in \SH(R)}\int_Q G(x) \,dx\\
	& \lesssim \sum_{R} \left(\sum_{S\in \SH(R)} \int_S {|f_R-f(y)|^q} \, dy \right)^{\frac{1}{q}} \frac{\int_R MG(\xi)\, d\xi}{\ell(R)^{s+\frac{d}{q}}}  
\end{align*}
and, using the H\"older inequality again, we get
\begin{align*}
\circled{3}
	& \lesssim \left(\sum_{R}  \left(\sum_{S\in \SH(R)} \int_S {|f_R-f(y)|^q} \, dy \right)^{\frac{p}{q}} \frac{\ell(R)^d}{\ell(R)^{sp+\frac{dp}q}} \right)^{\frac1p} \norm{MG}_{L^{p'}}.
\end{align*}
By the boundedness of the maximal operator in $L^{p'}$ we have that $\norm{MG}_{L^{p'}}\lesssim 1$. Now, given $R,S\in\mathcal{W}$ there exists an admissible chain $[S,R]$, and  we can decompose the previous expression as
\begin{align}\label{eqBreak3Final}
\circled{3}^p
	& \lesssim  \sum_{R} \left(\sum_{S\in \SH(R)} \left|\sum_{P\in[S,R)}\left(f_P-f_{\mathcal{N}(P)}\right)\frac{\ell(P)^{\frac{s}q}}{\ell(P)^{\frac{s}q}}\right|^q\ell(S)^d\right)^{\frac{p}{q}} \ell(R)^{d-sp-d\frac{p}{q} }\\
\nonumber	& \quad +  \sum_{R} \left(\sum_{S\in \SH(R)} \int_S {|f_S-f(y)|^q} \, dy \right)^{\frac{p}{q}} \ell(R)^{d-sp-d\frac{p}{q} }=:\circled{3.1}+\circled{3.2},
\end{align}
where we wrote $[S,R)=[S,R]\setminus\{R\}$.

Using H\"older's inequality 
\begin{align*}
\circled{3.1}
	& \lesssim  \sum_{R} \left(\sum_{S\in \SH(R)} \sum_{P\in[S,R)}\frac{|f_P-f_{\mathcal{N}(P)}|^q}{\ell(P)^{s}}\left(\sum_{P\in[S,R)}\ell(P)^{\frac{sq'}{q}} \right)^{\frac{q}{q'}} \ell(S)^d\right)^{\frac{p}{q}} \ell(R)^{d-sp-d\frac{p}{q} }.
\end{align*}
But for $S\in \SH(R)$ by Remark \ref{remInTheShadow} we have that  $ \sum_{P\in[S,R)}\ell(P)^{\frac{sq'}{q}}\lesssim\ell(R)^{\frac{sq'}{q}}$. Moreover, by \rf{eqLChainInShadowP} there exists a ratio $\rho_2$ such that for $P\in[S,R]$ we have that $S\in \SH^2(P):=\SH_{\rho_2}(P)$ and $P\in\SH^2(R)$. We also know that $\sum_{S\in \SH^2(P)}\ell(S)^d \lesssim\ell(P)^d$, so writing $U_P$ for the union of the neighbors of $P$, we get
\begin{align*}
\circled{3.1}
	& \lesssim  \sum_{R} \left( \sum_{P\in \SH^2(R)}\frac{\left(\fint_{U_P}|f(\xi)-f_{P}|\,d\xi\right)^q\ell(P)^d}{\ell(P)^{s}}\right)^{\frac{p}{q}} \ell(R)^{d+\frac{sp}{q}-sp-\frac{dp}q}.
\end{align*}
Recall that $p\geq q$ and, therefore, by H\"older's inequality and \rf{eqAscendingToGlory} we have that
\begin{align*}
\circled{3.1}
	& \lesssim  \sum_{R} \sum_{P\in \SH^2(R)} \frac{\left(\fint_{U_P}|f(\xi)-f_{P}|\,d\xi\right)^p \ell(P)^d}{\ell(P)^{\frac{sp}q}}  \left( \sum_{P\in \SH^2(R)} \ell(P)^d \right)^{\left(1-\frac{q}{p}\right) \frac{p}{q} } \ell(R)^{d-\frac{sp}{q'}-\frac{dp}q}\\
	& \lesssim  \sum_P \frac{\left(\fint_{U_P}|f(\xi)-f_{P}|\,d\xi\right)^p \ell(P)^d}{\ell(P)^{\frac{sp}q}}  \sum_{R: P\in \SH^2(R)}  \ell(R)^{-\frac{sp}{q'}} \approx  \sum_P \frac{\left(\fint_{U_P}|f(\xi)-f_{P}|\,d\xi\right)^p \ell(P)^d}{\ell(P)^{sp}}  
\end{align*}
Using Jensen's inequality we get
\begin{align}\label{eqKeyPoint4}
\circled{3.1}
	& \lesssim \sum_{P} \int_{U_P}  \frac{|f(\xi)-f_{P}|^p}{\ell(P)^{sp} }\,d\xi,
\end{align}
and Jensen's inequality again leads to 
\begin{align}\label{eq31finalfinal}
\circled{3.1}
	& \lesssim \sum_{P} \int_{U_P} \left(\frac{\int_{P} |f(\xi)-f(\zeta)|^q\,d\zeta}{\ell(P)^d} \right)^{\frac{p}{q}}\frac{1}{\ell(P)^{sp} }\,d\xi \lesssim \sum_{P} \int_{P} \left(\frac{\int_{5P} |f(\xi)-f(\zeta)|^q\,d\zeta}{|\xi-\zeta|^{sq+d}} \right)^{\frac{p}{q}}d\xi.
\end{align}

To bound $\circled{3.2}$ we follow the same scheme. Since $p \geq q$ we have that
\begin{align*}
\circled{3.2}
	& = \sum_{R} \left(\sum_{S\in \SH(R)} \int_S {|f_S-f(y)|^q} \, dy \frac{\ell(S)^{d\left(1-\frac{q}{p}\right)}}{\ell(S)^{d\left(1-\frac{q}{p}\right)}}\right)^{\frac{p}{q}} \ell(R)^{d-sp-d\frac{p}{q}}\\
	& \leq \sum_{R}\left(\sum_{S\in \SH(R)} \frac{ \left(\int_S {|f_S-f(y)|^q} \, dy\right)^{\frac{p}{q}}}{\ell(S)^{d\left(\frac{p}{q}-1\right)}}\right)^{\frac{q}{p}\cdot\frac{p}{q}}\left(\sum_{S\in \SH(R)} \ell(S)^{d} \right)^{\left(1-\frac{q}{p}\right)\frac{p}{q}} \ell(R)^{d-sp-d\frac{p}{q}},
\end{align*}
and, since $\sum_{S\in \SH(R)} \ell(S)^{d}\approx\ell(R)^d$, reordering and using \rf{eqAscendingToGlory} we get that 
\begin{align*}
\circled{3.2}
	& \lesssim \sum_{S} \frac{ \left(\int_S {|f_S-f(y)|^q} \, dy\right)^{\frac{p}{q}}}{\ell(S)^{d\left(\frac{p}{q}-1\right)}}  \sum_{R: S\in \SH(R)} \ell(R)^{-sp} \lesssim \sum_{S} \left(\frac{\int_S {|f_S-f(y)|^q} \, dy }{\ell(S)^{d}}\right)^{\frac{p}{q}}\frac{\ell(S)^{d}}{\ell(S)^{sp}}.
\end{align*}
Thus, by Jensen's inequality,
\begin{align*}
\circled{3.2}
	& \lesssim \sum_{S} \frac{\int_S {|f_S-f(y)|^p} \, dy }{\ell(S)^{d}}\frac{\ell(S)^{d}}{\ell(S)^{sp}}
\end{align*}
and, arguing as in \rf{eqKeyPoint4}, we get that
\begin{align} \label{eq32finalfinal}
\circled{3.2}
	& \lesssim \sum_{S} \int_S \left(\frac{\int_{S} |f(y)-f(\zeta)|^q\,d\zeta}{|y-\zeta|^{sq+d}} \right)^{\frac{p}{q}}dy.
\end{align}

Thus, by \rf{eq3Improved}, \rf{eqBreak3Final},  \rf{eq31finalfinal} and \rf{eq32finalfinal}, we have that
$$\sum_{Q,S}\int_Q\int_{S} \frac{|f(x)-f(y)|}{\Dist(Q,S)^{s+\frac{d}{q}}} g(x,y) \,dy \,dx
\lesssim \left(\sum_S \int_S \left( \int_{5S} \frac{|f(\xi)-f(y)|^q}{|\xi-y|^{sq+d}}\, dy \right)^{\frac{p}{q}}d\xi\right)^\frac1p.$$
This fact, together with \rf{eqSeminormDuality}, \rf{eqIntegralStrichartzClose} and \rf{eqStrichartzFarNear}  finishes the proof of  Lemma \ref{lemSecondReduction}.
\end{proof}

\begin{remark}
An analogous result to Lemma \ref{lemSecondReduction} for Besov spaces $B^s_{p,p}$ can be found in \cite[Proposition 5]{Dyda} where it is stated in the case of Lipschitz domains.
\end{remark}

\begin{corollary}\label{coroExtensionDomainFinal}
Let $\Omega$ be a uniform domain. Let $\delta(x):=\dist(x,\partial\Omega)$ for every $x\in\C$. 

Given $1<p<q<\infty$ and $0<s<1$ with $s>\frac{d}{p}-\frac{d}{q}$,  we have that
$A^s_{p,q}(\Omega)=F^s_{p,q}(\Omega)$ and, moreover, for $\rho_1>1$ big enough, we have that
$$\norm{f}_{F^s_{p,q}(\Omega)}\approx \norm{f}_{L^p(\Omega)}+\left(\int_\Omega \left(\int_{B_{\rho_1 \delta(x)}(x)\cap \Omega} \frac{|f(x)- f(y)|^q}{|x-y|^{sq+d}} \,dy\right)^\frac{p}{q}  \,dx\right)^{\frac{1}{p}} \mbox{\quad\quad for all } f\in F^s_{p,q}(\Omega).$$

Given $1<q\leq p<\infty$ and $0<s<1$,  we have that
$A^s_{p,q}(\Omega)=F^s_{p,q}(\Omega)$ and, moreover, for $0<\rho_0<1$ we have that
$$\norm{f}_{F^s_{p,q}(\Omega)}\approx \norm{f}_{L^p(\Omega)}+\left(\int_\Omega \left(\int_{B_{\rho_0\delta(x)}(x)} \frac{| f(x)- f(y)|^q}{|x-y|^{sq+d}} \,dy\right)^\frac{p}{q}  \,dx\right)^{\frac{1}{p}} \mbox{\quad\quad for all } f\in F^s_{p,q}(\Omega).$$
\end{corollary}
\begin{proof}
This comes straight forward from Corollary \ref{coroExtensionDomain}, Lemma \ref{lemFirstReduction} and Lemma \ref{lemSecondReduction}, taking smaller cubes in the Whitney covering if necessary when $\rho_0<<1$.
\end{proof}

\section{Calder\'on-Zygmund operators}\label{secT1}

\begin{definition}\label{defCZK}
We say that a measurable function $K \in L^1_{loc}(\R^d \setminus \{0\})$ is an \emph{admissible convolution Calder\'on-Zygmund kernel of order $\sigma\leq 1$} if it satisfies the size condition
\begin{equation}\label{eqCalderonZygmundSize}
|K(x)|\leq \frac{ C_K }{|x|^{d}} \mbox{\,\,\,\, for $x\neq 0$},
\end{equation}
and the H\"older smoothness condition
\begin{equation}\label{eqCalderonZygmundHolder}
|K(x-y)-K(x)|\leq \frac{ C_K |y|^\sigma}{|x|^{d+\sigma}} \mbox{\,\,\,\, for  $0<2|y|\leq |x|$},
\end{equation}
for a positive constant $C_K$ and that kernel can be extended to a convolution with a tempered distribution $W_K$ in $\R^d$ in the sense that for every Schwartz functions $f,g \in \mathcal{S}$ with $\supp (f)\cap \supp(g)=\emptyset$,  one has
\begin{equation}\label{eqCalderonDistribution}
\langle W_K * f, g \rangle=\int_{\R^d\setminus\{0\}} K(x) \left(f_-*g\right)(x) \, dx.
\end{equation}
 \end{definition}
  
 \begin{remark}
 We are using the notion of distributional convolution. Given Schwartz functions $f$ and $g$, the convolution coincides with multiplication at the Fourier side, that is, $f*g(x)=(\widehat{f} \cdot \widehat{g})\widecheck{\,}$. Given a tempered distribution $W$, a function $f\in\mathcal{S}$ and $x\in\R^d$, the tempered distribution $W*f$ is defined as   
$$\langle W*f , g \rangle :=\langle (\widehat{W} \cdot \widehat{f}) \widecheck{\,}, g \rangle =\langle \widehat{W} , \widehat{f}\cdot \widecheck{g} \rangle =\langle W , f_- * g \rangle \mbox{\quad\quad for every } g\in\mathcal{S}. $$
Note that $ f_- * g(x)=\int f(-y) g(x-y)\, dy$, so in case $\supp (f)\cap \supp(g)=\emptyset$ then $  f_- * g\equiv 0$ in a neighborhood of $0$ and, therefore, the integral in \rf{eqCalderonDistribution} is absolutely convergent by \rf{eqCalderonZygmundSize}. 

In any case, the distribution $W* f$ is regular (i.e., it can be expressed as an $L^1_{loc}$ function) and it coincides with the $C^\infty$ function $W* f(x)=\langle W, \tau_x f_-\rangle$, where $\tau_x f_- (y)=f_-(y-x)$ (see \cite[Chapter I, Theorem 3.13]{SteinWeiss}).
\end{remark}

 There are some cancellation conditions that one can impose to a kernel satisfying the size condition \rf{eqCalderonZygmundSize} to grant that it can be extended to a convolution with a tempered distribution. For instance, if $K$ satisfies \rf{eqCalderonZygmundSize} and $W_K$ is a principal value  operator in the sense that
 \begin{equation}\label{eqPrincipalValue}
 \langle W_K,\varphi \rangle = \lim_{j\to\infty} \int_{|x|\geq \delta_j} K(x)\varphi(x)\, dx \mbox{\quad\quad for all }\varphi \in \mathcal{S},
\end{equation}
for a certain sequence $\delta_j \searrow 0$, then $W_K$ satisfies \rf{eqCalderonDistribution} (see \cite[Section 4.3.2]{Grafakos}).

\begin{definition}\label{defCZO}
We say that an operator $T:\mathcal{S} \to \mathcal{S'}$ is a \emph{convolution Calder\'on-Zygmund operator of order $\sigma\in(0,1]$} with kernel $K$ if 
\begin{enumerate}
\item $K$ is an admissible convolution Calder\'on-Zygmund kernel of order $\sigma$ which can be extended to a convolution with a tempered distribution $W_K$,
\item $T$ satisfies that $T f= W_K * f$ for all $f\in\mathcal{S}$ and
\item $T$ extends to an operator bounded in $L^2$.
\end{enumerate}
 \end{definition}

\begin{remark}\label{remMultipliers}
Using the Calder\'on-Zygmund decomposition one can see that $T$ is also bounded on $L^p$ for $1<p<\infty$ (see \cite[proof of Theorem 4.3.3]{Grafakos}). Thus,  the Fourier transform of a convolution Calder\'on-Zygmund operator $T$ is a Fourier multiplier for $L^p$. We refer the reader to \cite[Section 2.6]{TriebelTheory}. 

These operators are bounded in $L^p(w)$ for every $w\in A_p$ (see \cite[Definition 5.11 and Theorem 7.11]{Duoandikoetxea}, for instance). Now, \cite[Theorem 10.17 combined with Section 12]{FrazierJawerth} grants that they are Fourier multipliers for $F^0_{p,q}$ for every pair $1<p,q<\infty$ as well.

Therefore, such an operator is always a Fourier multiplier for  $F^0_{p,q}$.  But being a Fourier multiplier for $F^0_{p,q}$ implies being a Fourier multiplier also for $F^s_{p,q}$ for every $s$ (see \cite[Section 2.6]{TriebelTheory}).
\end{remark}

It is a well-known fact that the Schwartz class is dense in $L^p$ for $1 \leq p<\infty$. Thus, if $f\in L^p$ and $x\notin \supp(f)$, then
\begin{equation}\label{eqDefinitionCZ}
Tf(x) = \int K(x-y)f(y) dy.
\end{equation}

To prove Theorem \ref{theoT1}  we need the following lemma which says that it is equivalent to bound the transform of a function and its approximation by constants on Whitney cubes.

To do so, we define the fractional derivative, 
\begin{definition}\label{defFractionalGradient}
Given a uniform domain $\Omega$ with Whitney covering $\mathcal{W}$ and $f\in L^p(\Omega)$ for certain values $0<s<1$ and $1<q<\infty$, the $s$-th fractional gradient of index $q$ of $f$ in a point $x\in Q\in \mathcal{W}$ is
\begin{equation*}
\nabla^s_q f(x):=\left(\int_{\Sh(Q)}\frac{|f(x)-f(y)|^q}{|x-y|^{sq+d}} \,dy\right)^{\frac{1}{q}}.
\end{equation*}
Then, by Corollary \ref{coroExtensionDomain} and Lemma \ref{lemFirstReduction}, for $1<p<\infty$ with $\frac{d}{p}-\frac{d}{q}<s$, we have that
\begin{equation}\label{eqNormShadow}
\norm{f}_{F^s_{p,q}(\Omega)}\approx\norm{f}_{L^p(\Omega)}+\norm{\nabla^s_q f}_{L^p(\Omega)}.
\end{equation}
\end{definition}

The following result is the key to Theorem 1.1. Recall that $T_\Omega(f)=\chi_\Omega \, T(\chi_\Omega \, f)$. Note that $\chi_\Omega$ is not in $L^p$ if $\Omega$ is unbounded. However, $\nabla^s_q T_\Omega 1(x)$ can be defined for $x,y\in \Omega$ using a bump function $\varphi_{xy}$, compactly supported in $\Omega$ with value 1 in an open set containing both of them
\begin{equation}\label{eqInfiniteSupport}
T\chi_\Omega(x)-T\chi_\Omega(y):=T \varphi_{xy} (x)-T \varphi_{xy} (y)+\int_\Omega((1-\varphi_{xy}(w))(K(x-w)-K(y-w))\, dm(w),
\end{equation}
which is well defined by \rf{eqCalderonZygmundHolder} and does not depend on the particular choice of $\varphi_{xy}$.

\begin{klemma}\label{lemKLemmaT1}
Let $\Omega$ be a uniform domain with Whitney covering $\mathcal{W}$, let $T$ be a convolution Calder\'on-Zygmund operator of order $0<s< 1$, $1<p<\infty$ and $1<q<\infty$ with $s>\frac{d}{p}-\frac{d}{q}$. 
The following statements are equivalent:
\begin{enumerate}[i)]
\item For every $f\in F^s_{p,q}(\Omega)$ one has
$$\norm{T_\Omega f}_{F^s_{p,q} (\Omega)}\leq C  \norm{f}_{F^s_{p,q}(\Omega)},$$
with $C$ independent from $f$.
\item For every  $f\in F^s_{p,q}(\Omega)$ one has
$$\sum_{Q\in\mathcal{W}} |f_Q|^p \norm{\nabla^s_q T \chi_\Omega}_{L^p(Q)}^p\leq C \norm{f}^p_{F^s_{p,q}(\Omega)},$$
with $C$ independent from $f$.
\end{enumerate}
Moreover, 
\begin{equation}\label{eqObjective1}
\sum_{Q\in\mathcal{W}} \norm{ \nabla^s_q T_\Omega (f-f_Q)}_{L^p(Q)}^p \lesssim \left(C_K +  \norm{T}_{F^s_{p,q}\to F^s_{p,q}}+  \norm{T}_{L^p\to L^p}+\norm{T}_{L^q\to L^q}\right)^p \norm{f}^p_{F^s_{p,q}(\Omega)}.
\end{equation}
\end{klemma}

\begin{proof}
Let $\Omega$ be an $\varepsilon$-uniform domain. The core of the proof is showing that \rf{eqObjective1} holds. Once this is settled,  since we have that 
\begin{align*}
 \sum_{Q\in\mathcal{W}} \norm{ \nabla^s_q T_\Omega f}_{L^p(Q)}^p
	& \lesssim_p \sum_{Q\in\mathcal{W}} \norm{ \nabla^s_q T_\Omega (f-f_Q)}_{L^p(Q)}^p +\sum_{Q\in\mathcal{W}} |f_Q|^p \norm{ \nabla^s_q T_\Omega 1}_{L^p(Q)}^p  ,
\end{align*}
and
\begin{align*}
\sum_{Q\in\mathcal{W}} |f_Q|^p \norm{ \nabla^s_q T_\Omega 1}_{L^p(Q)}^p  
	& \lesssim_p  \sum_{Q\in\mathcal{W}} \norm{ \nabla^s_q T_\Omega (f_Q-f)}_{L^p(Q)}^p + \sum_{Q\in\mathcal{W}} \norm{ \nabla^s_q T_\Omega f}_{L^p(Q)}^p,
\end{align*}
inequality \rf{eqObjective1} proves that
$$ \sum_{Q\in\mathcal{W}} \norm{ \nabla^s_q T_\Omega f}_{L^p(Q)}^p\lesssim  \norm{f}_{F^s_{p,q}(\Omega)}^p \iff \sum_{Q\in\mathcal{W}} |f_Q|^p \norm{ \nabla^s_q T_\Omega 1}_{L^p(Q)}^p   \lesssim  \norm{f}_{F^s_{p,q}(\Omega)}^p.$$
On the other hand, by assumption $T$ is bounded on $L^p$ and we have that $\norm{T_\Omega f}_{L^p (\Omega)}\lesssim  \norm{f}_{L^p(\Omega)}$. Since
$\norm{T_\Omega f}_{F^{s}_{p,q}(\Omega)}^p \approx \norm{T_\Omega f}_{L^p(\Omega)}^p + \int_\Omega | \nabla^s_q T_\Omega f (x)|^p \,dx$ by \rf{eqNormShadow},
the lemma follows.

Again we use duality. That is, to prove \rf{eqObjective1} it suffices to prove that given a positive function $g\in L^{p'}(L^{q'}(\Omega))$ with $\norm{g}_{ L^{p'}(L^{q'}(\Omega))}= 1$, we have that
$$\sum_Q\int_Q  \int_{\Sh(Q)} \frac{\left|T_\Omega \left(f-f_Q\right)(x)-T_\Omega \left(f-f_Q\right)(y)\right|}{|x-y|^{s+\frac{d}{q}}} g(x,y)\, dy\, dx  \lesssim \norm{f}_{F^s_{p,q}(\Omega)}.$$

Given a cube $Q\in \mathcal{W}$, we can define a bump function $\varphi_Q$ such that $\chi_{6Q}\leq \varphi_Q\leq \chi_{7Q}$ and $\norm{\nabla \varphi_Q}_{L^\infty} \leq C \ell(Q)^{-1}$. Given a cube $S\subset 5Q$ we define $\varphi_{QS}:=\varphi_Q$. Otherwise, take $\varphi_{QS}:=\varphi_S$. Note that in both situations, by \rf{eqWhitney5} we have that $\supp\varphi_{QS}\subset 23S$. Then, we can express the difference between $T_\Omega (f-f_Q)$ evaluated at $x\in Q$ and in $y\in S$ as
\begin{align}\label{eqBreakQS}
T_\Omega (f-f_Q)(x)-T_\Omega (f-f_Q)(y)
			& =T_\Omega \left[\left(f-f_Q\right)\varphi_Q\right](x)-T_\Omega \left[\left(f-f_Q\right)\varphi_{QS}\right](y)\\
\nonumber	& \quad +T_\Omega \left[\left(f-f_Q\right)\left(1-\varphi_Q\right)\right](x)-T_\Omega \left[\left(f-f_Q\right)\left(1-\varphi_{QS}\right)\right](y),
\end{align}
where all the terms must be understood  in the sense of \rf{eqInfiniteSupport}. Note that the first two terms in the right-hand side of \rf{eqBreakQS} are `local' terms in the sense that the functions to which we apply the operator $T_\Omega$ are supported in a small neighborhood of the point of evaluation (and  are globally $F^s_{p,q}$, as we will check later on) and the other two terms are `non-local'. What we will prove is that the local part 
$$\squared{1}:=\sum_Q\int_Q \sum_{S\in\SH(Q)} \int_S \frac{\left|T_\Omega \left[\left(f-f_Q\right)\varphi_Q\right](x)-T_\Omega \left[\left(f-f_Q\right)\varphi_{QS}\right](y)\right|}{|x-y|^{s+\frac{d}{q}}} g(x,y)\, dy\, dx ,$$
and the non-local part
$$\squared{2}:= \sum_Q\int_Q \sum_{S\in\SH(Q)} \int_S \frac{\left|T_\Omega \left[\left(f-f_Q\right)\left(1-\varphi_Q\right)\right](x)-T_\Omega \left[\left(f-f_Q\right)\left(1-\varphi_{QS}\right)\right](y)\right|}{|x-y|^{s+\frac{d}{q}}} g(x,y)\, dy\, dx ,$$
 are both bounded as
\begin{equation}\label{eqObjective2}
\squared{1}+\squared{2} \leq C \norm{f}_{F^s_{p,q}(\Omega)}.
\end{equation}

We begin by the local part, that is, we want to prove that $\squared{1} \lesssim \norm{f}_{F^s_{p,q}(\Omega)}$.
Note that for $x\in Q$ and $y \in S \in \SH(Q)$, if $y\in3Q$ then $\varphi_{QS}=\varphi_Q$ and, otherwise $|x-y|\approx \ell(Q)$. Thus,
\begin{align}\label{eqBreak2}
\squared{1}
			& \leq\sum_Q\int_Q \int_{3Q} \frac{\left|T \left[\left(f-f_Q\right)\varphi_Q\right](x)-T \left[\left(f-f_Q\right)\varphi_{Q}\right](y)\right|}{|x-y|^{s+\frac{d}{q}}} g(x,y)\, dy\, dx \\
\nonumber	& \quad +\sum_Q\int_Q \int_{\Sh(Q)} \frac{\left|T \left[\left(f-f_Q\right)\varphi_Q\right](x)\right|}{\ell(Q)^{s+\frac{d}{q}}} g(x,y)\, dy\, dx  \\
\nonumber	& \quad +\sum_Q\int_Q \sum_{S\in \SH(Q)} \int_S \frac{\left|T \left[\left(f-f_Q\right)\varphi_{QS}\right](y)\right|}{\ell(Q)^{s+\frac{d}{q}}} g(x,y)\, dy\, dx = : \squared{1.1} + \squared{1.2} + \squared{1.3}.
\end{align}

Of course, by H\"older's inequality we have that
\begin{align*}
\squared{1.1}^p
	& \leq \sum_Q\int_Q \left(\int_{3Q} \frac{\left|T \left[\left(f-f_Q\right)\varphi_Q\right](x)-T \left[\left(f-f_Q\right)\varphi_{Q}\right](y)\right|^q}{|x-y|^{sq+d}}\, dy\right)^{\frac{p}{q}}dx \norm{g}_{L^{p'}(L^{q'}(\Omega))}^{p}.
\end{align*}
By Corollary \ref{coroExtensionDomain} we get
\begin{align*}
\squared{1.1}^p
	& \lesssim  \sum_Q\norm{T \left[\left(f-f_Q\right)\varphi_Q\right]}_{F^s_{p,q}(\R^d)}^p.
\end{align*}
Now, the operator $T$ is bounded on $F^s_{p,q}$ by assumption (see Definition \ref{defCZO} and Remark \ref{remMultipliers}).
Thus,
$$\squared{1.1}^p \lesssim \norm{T}_{F^s_{p,q}\to F^s_{p,q}}^p \sum_Q \norm{ \left(f-f_Q\right)\varphi_Q }_{F^s_{p,q}(\R^d)}^p.$$
Consider the characterization of the  $F^s_{p,q}$-norm given in Corollary \ref{coroCharacterizationRd}.  Since $\varphi_Q\leq \chi_{7Q}$, the first term $\sum_Q \norm{ \left(f-f_Q\right)\varphi_Q }_{L^p(\R^d)}^p$ is bounded by a constant times $\norm{f}_{L^p}$ by the finite overlapping of the Whitney cubes and the Jensen inequality, and the second is 
 $$\sum_Q\int_{\R^d} \left(\int_{\R^d} \frac{\left|\left(f(x)-f_Q\right) \varphi_Q(x)-\left(f(y)-f_Q\right)\varphi_Q(y)\right|^q}{|x-y|^{sq+d}}\, dy\right)^{\frac{p}{q}}dx,$$
  where the integrand vanishes if both $x,y\notin 8Q$. Therefore, we can write
\begin{align}\label{eqBreak21}
\squared{1.1}^p
\nonumber	& \lesssim \norm{f}_{L^p}+\sum_Q \int_{8Q} \left(\int_{8Q} \frac{\left|\left(f(x)-f_Q\right) \varphi_Q(x)-\left(f(y)-f_Q\right)\varphi_Q(y)\right|^q}{|x-y|^{sq+d}}\, dy\right)^{\frac{p}{q}}dx\\
			& \quad + \sum_Q \int_{\R^d\setminus 8Q} \left(\int_{7Q} \frac{\left|\left(f(y)-f_Q\right)\varphi_Q(y)\right|^q}{|x-y|^{sq+d}}\, dy\right)^{\frac{p}{q}}dx\\
\nonumber	& \quad + \sum_Q \int_{7Q} \left(\int_{\R^d\setminus 8Q} \frac{\left|\left(f(x)-f_Q\right) \varphi_Q(x)\right|^q}{|x-y|^{sq+d}}\, dy\right)^{\frac{p}{q}}dx =: \norm{f}_{L^p}+ \squared{1.1.1}+\squared{1.1.2}+\squared{1.1.3},
\end{align}
where the constant depends linearly on the operator norm $\norm{T}_{F^s_{p,q}\to F^s_{p,q}}^p$.

Adding and subtracting $\left(f(x)-f_Q\right) \varphi_Q(y)$ in the numerator of the integral in $\squared{1.1.1}$ we get that
\begin{align*}
\squared{1.1.1}
			& \lesssim \sum_Q \int_{8Q}\left(\int_{8Q} \frac{\left|f(x)-f_Q\right|^q \left|\varphi_Q(x)-\varphi_Q(y)\right|^q}{|x-y|^{sq+d}}\, dy\right)^{\frac{p}{q}}dx\\
			& \quad + \sum_Q \int_{8Q}  \left(\int_{8Q} \frac{\left|f(x)-f(y)\right|^q }{|x-y|^{sq+d}}\, dy\right)^{\frac{p}{q}}dx.
\end{align*}
The second term is bounded by a constant times $\norm{f}_{F^s_{p,q}(\Omega)}^p$, so
\begin{align*}
\squared{1.1.1}
			& \lesssim \sum_Q \int_{8Q}\left(\int_{8Q} \frac{\norm{\nabla\varphi_Q}_{L^\infty}^q|x-y|^q}{|x-y|^{sq+d}}\, dy\right)^{\frac{p}{q}}\left|f(x)-f_Q\right|^p \, dx + \norm{f}_{F^s_{p,q}(\Omega)}^p.
\end{align*}
Using $\norm{\nabla\varphi_Q}_{L^\infty}\lesssim\frac{1}{\ell(Q)}$ and the local inequality for the maximal operator \rf{eqMaximalClose} we get that
\begin{align}\label{eq211Final}
\squared{1.1.1}
			& \lesssim \sum_Q \int_{8Q}\ell(Q)^{(1-s)p} \frac{\left|f(x)-f_Q\right|^p}{\ell(Q)^p} \, dx + \norm{f}_{F^s_{p,q}(\Omega)}^p\\
\nonumber	& \lesssim \sum_Q \int_{8Q} \left(\frac{\int_Q\left|f(x)-f(\xi)\right|\, d\xi}{\ell(Q)^{s+d}}\right)^p \, dx + \norm{f}_{{F}^s_{p,q}(\Omega)}^p.
\end{align}
By Jensen's inequality $\frac{1}{\ell(Q)^d}\int_Q\left|f(x)-f(\xi)\right|\, d\xi\lesssim \left(\int_Q\frac{1}{\ell(Q)^d}\left|f(x)-f(\xi)\right|^q\, d\xi\right)^\frac{1}{q}$ and, therefore, 
\begin{align}\label{eq211}
\squared{1.1.1} \lesssim \norm{f}_{F^s_{p,q}(\Omega)}^p.
\end{align}

Now we undertake the task of bounding $\squared{1.1.2}$ in \rf{eqBreak21}. Writing $x_Q$ for the center of a given cube $Q$, we have that
$$\squared{1.1.2}\lesssim \sum_Q \int_{\R^d\setminus 8Q} \frac{dx}{|x-x_Q|^{sp+\frac{dp}{q}}}\left(\int_{7Q} \left|f(y)-f_Q\right|^q\, dy\right)^{\frac{p}{q}}.$$
Since $s>\frac{d}{p}-\frac{d}{q}$ we have that $sp+\frac{dp}{q}>d$. Thus, by \rf{eqMaximalAllOver}
$$\squared{1.1.2}\lesssim \sum_Q  \frac{1}{\ell(Q)^{sp+\frac{dp}{q}-d}}\left(\int_{7Q} \left|f(y)-f_Q\right|^q\, dy\right)^{\frac{p}{q}}\leq \sum_Q \frac{\left(\int_{7Q}  \left(\int_Q\left|f(y)-f(\xi)\right|\,d\xi\right)^q\, dy\right)^{\frac{p}{q}}}{\ell(Q)^{sp+\frac{dp}{q}-d+dp}}.$$
By Minkowski's inequality we have that
$$\squared{1.1.2}\lesssim \sum_Q \frac{\left(\int_{Q}  \left(\int_{7Q} \left|f(y)-f(\xi)\right|^q\,dy\right)^\frac{1}{q}\, d\xi\right)^{p}}{\ell(Q)^{sp+\frac{dp}{q}+d(p-1)}},$$
and by H\"older's inequality, using that $p-1=\frac{p}{p'}$ we get that
$$\squared{1.1.2}\lesssim \sum_Q \frac{\int_{Q}  \left(\int_{7Q} \left|f(y)-f(\xi)\right|^q\,dy\right)^\frac{p}{q} d\xi \ell(Q)^{\frac{dp}{p'}}}{\ell(Q)^{sp+\frac{dp}{q}+\frac{dp}{p'}}}\lesssim \sum_Q  \int_{Q}  \left(\int_{7Q} \frac{\left|f(y)-f(\xi)\right|^q\,dy}{|y-\xi|^{sq+d}}\right)^\frac{p}{q} \, d\xi$$
and
\begin{align}\label{eq212}
\squared{1.1.2} \lesssim \norm{f}_{F^s_{p,q}(\Omega)}^p.
\end{align}

Dealing with the last term in \rf{eqBreak21} is somewhat easier. Note that by \rf{eqMaximalAllOver} we have that
$$\squared{1.1.3}\leq \sum_Q \int_{7Q} \left|f(x)-f_Q\right|^p\left(\int_{\R^d\setminus 8Q} \frac{1}{|x-y|^{sq+d}}\, dy\right)^{\frac{p}{q}}dx\leq \sum_Q \int_{7Q} \frac{\left|f(x)-f_Q\right|^p }{\ell(Q)^{sp}} \, dx $$
and, since this quantity is bounded by the right-hand side of \rf{eq211Final}, it follows that 
\begin{align}\label{eq213}
\squared{1.1.3} \lesssim \norm{f}_{F^s_{p,q}(\Omega)}^p.
\end{align}
Summing up, by \rf{eqBreak21}, \rf{eq211}, \rf{eq212} and \rf{eq213} we get
\begin{align}\label{eq21Done}
\squared{1.1} \lesssim \norm{T}_{F^s_{p,q}\to F^s_{p,q}}\norm{f}_{F^s_{p,q}(\Omega)}.
\end{align}

Back to \rf{eqBreak2}, it remains to bound $\squared{1.2}$ and $\squared{1.3}$. Recall that
$$\squared{1.2}=\sum_Q\int_Q  \frac{\left|T \left[\left(f-f_Q\right)\varphi_Q\right](x)\right|}{\ell(Q)^{s+\frac{d}{q}}} \int_{\Sh(Q)}g(x,y)\, dy\, dx.$$
Writing $G(x)=\norm{g(x,\cdot)}_{L^{q'}(\Omega)}$ and using H\"older's inequality we get
$$\int_{\Sh(Q)} g(x,y) \, dy\leq \left(\int_{\Sh(Q)} g(x,y)^{q'} \, dy\right)^{\frac{1}{q'}}|\Sh(Q)|^{\frac{1}{q}}\lesssim_{\rho_\varepsilon,d} G(x) \ell(Q)^{\frac{d}{q}},$$
and using again H\"older's inequality it follows that
\begin{align*}
\squared{1.2}
	& \lesssim \left(\sum_Q\int_Q  \frac{\left| T \left[\left(f-f_Q\right)\varphi_Q\right](x)  \right|^p}{\ell(Q)^{sp}}\, dx\right)^\frac{1}{p}\norm{G}_{L^{p'}(\Omega)}.
\end{align*}
Of course, $\norm{G}_{L^{p'}(\Omega)}\leq 1$.
Now, by Definition \ref{defCZO} we can use the boundedness of $T$ in $L^p$ to find that
$$\squared{1.2}\lesssim \norm{T}_{L^p\to L^p} \left(\sum_Q  \frac{\norm{ \left(f-f_Q\right)\varphi_Q}_{L^p(\R^d)}^p}{\ell(Q)^{sp}}\right)^{\frac{1}{p}} \lesssim \left(\sum_Q  \frac{\norm{ f-f_Q }_{L^p(7Q)}^p}{\ell(Q)^{sp}}\right)^{\frac{1}{p}},$$
and we can argue again as in \rf{eq211Final} to prove that
\begin{align}\label{eq22Done}
\squared{1.2} \lesssim \norm{T}_{L^p\to L^p} \norm{f}_{F^s_{p,q}(\Omega)}.
\end{align}

Finally, for the last term in \rf{eqBreak2}, that is, for
$$\squared{1.3}=\sum_Q\int_Q \sum_{S\in \SH(Q)} \int_S \frac{\left|T \left[\left(f-f_Q\right)\varphi_{QS}\right](y)\right|}{\ell(Q)^{s+\frac{d}{q}}} g(x,y)\, dy\, dx,$$
by H\"older's inequality we have that
$$\squared{1.3}\leq \sum_Q\int_Q \left( \sum_{S\in \SH(Q)} \int_S \frac{\left|T \left[\left(f-f_Q\right)\varphi_{QS}\right](y)\right|^q}{\ell(Q)^{sq+d}} \, dy\right)^\frac{1}{q} G(x) \, dx.$$
The boundedness of $T$ in $L^q$ leads to
$$\squared{1.3}\leq \norm{T}_{L^q\to L^q}  \sum_Q \left( \sum_{S\in \SH(Q)} \int_{\supp(\varphi_{QS})} \frac{\left|\left(f(y)-f_Q\right)\varphi_{QS}(y)\right|^q}{\ell(Q)^{sq+d}} \, dy\right)^\frac{1}{q}\ell(Q)^d\inf_Q MG .$$
Given a cube $Q$, the finite overlapping of the family $\{50S\}_{S\in\mathcal{W}}$ (see Definition \ref{defWhitney}) implies the finite  overlapping of the supports of the family $\{\varphi_{QS}\}$ (recall that $\supp(\varphi_{QS})\subset 23S$), so there is a certain ratio $\rho_3$ such that naming $\Sh^3(Q):=\Sh_{\rho_3}(Q)$ we have that
\begin{align*}
\squared{1.3}
	& \lesssim \sum_Q \left(\int_{\Sh^3(Q)} \frac{\left|f(y)-f_Q\right|^q}{\ell(Q)^{sq+d-dq}} \, dy\right)^\frac{1}{q}\inf_Q MG\\
	& \leq \sum_Q \left(\int_{\Sh^3(Q)} \left(\int_Q\frac{\left|f(y)-f(\xi)\right|}{\ell(Q)^{s+\frac{d}{q}-d+d}}\,d\xi\right)^q \, dy\right)^\frac{1}{q} \inf_Q MG.
\end{align*}
Finally, using Minkowski's inequality and H\"older's inequality we get that
$$\squared{1.3}\lesssim  \sum_Q \int_Q\left(\int_{\Sh^3(Q)}\frac{\left|f(y)-f(\xi)\right|^q}{\ell(Q)^{sq+d}}\, dy\right)^\frac{1}{q}MG(\xi) \,d\xi\lesssim  \left(\sum_Q \int_Q\left(\int_{\Omega}\frac{\left|f(y)-f(\xi)\right|^q}{|x-y|^{sq+d}}\, dy\right)^\frac{p}{q}d\xi\right)^\frac{1}{p},$$
that is, 
\begin{align}\label{eq23Done}
\squared{1.3} \lesssim  \norm{T}_{L^q\to L^q} \norm{f}_{F^s_{p,q}(\Omega)}.
\end{align}

Now, by \rf{eqBreak2}, \rf{eq21Done},  \rf{eq22Done} and  \rf{eq23Done} we have that 
\begin{align}\label{eq2done}
\squared{1} \lesssim \left(  \norm{T}_{F^s_{p,q}\to F^s_{p,q}}+  \norm{T}_{L^p\to L^p}+\norm{T}_{L^q\to L^q}\right)\norm{f}_{F^s_{p,q}(\Omega)},
\end{align}
and we have finished with the local part.

Now we bound the non-local part in \rf{eqObjective2}. Consider $x\in Q\in\mathcal{W}$. By \rf{eqDefinitionCZ} (and \rf{eqInfiniteSupport} for unbounded domains), since $x$ is not in the support of $\left(f-f_Q\right)\left(1-\varphi_Q\right)$, we have that 
$$T_\Omega \left[\left(f-f_Q\right)\left(1-\varphi_Q\right)\right](x)=\int_\Omega K(x-z)\left(f(z)-f_Q\right)\left(1-\varphi_{Q}(z)\right) \, dm(z),$$
and by the same token for $y\in S\in\SH(Q)$
$$T_\Omega \left[\left(f-f_Q\right)\left(1-\varphi_{QS}\right)\right](y)=\int_\Omega K(y-z)\left(f(z)-f_Q\right)\left(1-\varphi_{QS}(z)\right) \, dm(z).$$
To shorten the notation, we will write
$$\lambda_{QS}(z_1,z_2)=K(z_1-z_2)\left(f(z_2)-f_Q\right)\left(1-\varphi_{QS}(z_2)\right),$$
for $z_1\neq z_2$. 
Then we have that
$$\Big|T_\Omega \left[\left(f-f_Q\right)\left(1-\varphi_Q\right)\right](x)-T_\Omega \left[\left(f-f_Q\right)\left(1-\varphi_{QS}\right)\right](y)\Big|=\left|\int_\Omega \left(\lambda_{QQ}(x,z)- \lambda_{QS}(y,z)\right)\, dm(z)\right|,$$
that is, 
\begin{align*}
\squared{2}
			& = \sum_Q \int_Q \sum_{S\in \SH(Q)}\int_S \frac{\left|\int_{\Omega}\left(\lambda_{QQ}(x,z)-\lambda_{QS}(y,z)\right)\,dz\right|}{|x-y|^{s+\frac{d}{q}}}g(x,y) \, dy\, dx.
\end{align*}

For $\rho_4$ big enough, $\Sh^4(Q):=\Sh_{\rho_4}(Q)\supset\bigcup_{S\in\SH(Q)}\Sh(S)$ (call $\SH^4(Q):=\SH_{\rho_4}(Q)$), we can decompose
\begin{align}\label{eq1FirstApproach}
\squared{2}
			&  \leq \sum_Q \int_Q \sum_{S\subset\Sh(Q)\setminus 2Q}\int_S \frac{\int_{\Sh^4(Q)}\left|\lambda_{QQ}(x,z)-\lambda_{QS}(y,z)\right|\,dz}{|x-y|^{s+\frac{d}{q}}}g(x,y) \, dy\, dx\\
\nonumber	& \quad + \sum_Q \int_Q \sum_{S\subset\Sh(Q)\setminus 2Q} \int_S \frac{\int_{\Omega\setminus\Sh^4(Q)}\left|\lambda_{QQ}(x,z)-\lambda_{QS}(y,z)\right|\,dz}{|x-y|^{s+\frac{d}{q}}}g(x,y) \, dy\, dx\\
\nonumber	& \quad + \sum_Q \int_Q \int_{5Q} \frac{\int_{\Omega}\left|\lambda_{QQ}(x,z)-\lambda_{QQ}(y,z)\right|\,dz}{|x-y|^{s+\frac{d}{q}}}g(x,y) \, dy\, dx=:\squared{A}+\squared{B}+\squared{C}.
\end{align}

In the first term in the right-hand side of \rf{eq1FirstApproach} the variable $z$ is `close' to either $x$ or $y$, so smoothness does not help. Thus, we will take absolute values, giving rise to two terms separating $\lambda_{QQ}$ and $\lambda_{QS}$. That is, we use that
$$\squared{A}\leq \sum_Q \int_Q \sum_{S\subset\Sh(Q)\setminus 2Q}\int_S \frac{\int_{\Sh^4(Q)}\left(\left|\lambda_{QQ}(x,z)\right|+\left|\lambda_{QS}(y,z)\right|\right)\,dz}{|x-y|^{s+\frac{d}{q}}}g(x,y) \, dy\, dx.$$
Using the size condition \rf{eqCalderonZygmundSize},
$$\left|\lambda_{QQ}(x,z)\right| \leq C_K \frac{|f(z)-f_Q|}{|x-z|^{d}}|1-\varphi_Q(z)|$$
and
$$\left|\lambda_{QS}(y,z)\right|\leq C_K \frac{|f(z)-f_Q|}{|y-z|^{d}}|1-\varphi_{QS}(z)|.$$

Summing up, 
\begin{align}\label{eq1SecondApproach}
\squared{A}
			& \lesssim_{C_K}  \sum_Q \int_Q \int_{\Sh(Q)\setminus 2Q} \int_{\Sh^4(Q)}\frac{\left|f(z)-f_Q\right| |1-\varphi_Q(z)|\,dz}{|x-y|^{s+\frac{d}{q}}|x-z|^{d}}g(x,y) \, dy\, dx\\
\nonumber			& \quad +  \sum_Q \int_Q \sum_{S\subset\Sh(Q)\setminus 2Q}\int_S \int_{\Sh^4(Q)}\frac{\left|f(z)-f_Q\right||1-\varphi_{QS}(z)|\,dz}{|x-y|^{s+\frac{d}{q}}|y-z|^d}g(x,y) \, dy\, dx =:\squared{2.1}+\squared{2.2},
\end{align}
with constants depending linearly on the Calder\'on-Zygmund constant $C_{K}$.

We begin by the shorter part, that is
$$\squared{2.1}=\sum_Q \int_Q \int_{\Sh(Q)\setminus 2Q} \int_{\Sh^4(Q)}\frac{\left|f(z)-f_Q\right| |1-\varphi_Q(z)|\,dz}{|x-y|^{s+\frac{d}{q}}|x-z|^{d}}g(x,y) \, dy\, dx.$$
Using the fact that $1-\varphi_{Q}(z)=0$ when $z$ is close to the cube $Q$,  we can say that 
$$\squared{2.1}\lesssim \sum_Q \frac{1}{\ell(Q)^{s+\frac{d}{q}+d}} \int_{\Sh^4(Q)\setminus 6Q} \left|f(z)-f_Q\right|\int_Q\int_{\Sh(Q)\setminus 2Q} g(x,y) \, dy\, dx\,dz .$$
Now, by the H\"older inequality we have that
$$\int_{\Sh(Q)\setminus 2Q} g(x,y) \, dy \lesssim_{\rho_\varepsilon,d} G(x) \ell(Q)^{\frac{d}{q}},$$
where $G(x)=\norm{g(x,\cdot)}_{L^{q'}}$.
Thus,
\begin{equation*}
\squared{2.1}\lesssim \sum_Q \int_{\Sh^4(Q)}\frac{ \left|f(z)-f_Q\right|}{\ell(Q)^{s+d}} \int_Q G(x)\, dx\,dz \lesssim \sum_Q \int_Q \int_{\Sh^4(Q)} \frac{ \left|f(z)-f(\xi)\right|}{\ell(Q)^{s+d}}  MG(\xi) \,dz\,d\xi .
\end{equation*}
Finally, by Jensen's inequality and the boundedness of the maximal operator in $L^{p'}$ we have that
\begin{align}\label{eqBoundDifferenceAndMaximal}
\sum_Q \int_Q \int_{\Sh^4(Q)} \frac{ \left|f(z)-f(\xi)\right|}{\ell(Q)^{s+d}}  MG(\xi) \,dz\,d\xi
	&\lesssim \sum_Q \int_Q\left( \int_{\Sh^4(Q)} \frac{ \left|f(z)-f(\xi)\right|^q}{\ell(Q)^{sq+d}}  \,dz\right)^\frac{1}{q}  MG(\xi)\,d\xi\\
\nonumber	&  \lesssim \left(\int_\Omega \left(\int_\Omega \frac{ \left|f(z)-f(\xi)\right|^q}{|z-\xi|^{sq+d}}  \,dz\right)^\frac{p}{q}d\xi\right)^\frac{1}{p} \norm{MG}_{L^{p'}}  ,
\end{align}
that is,
\begin{equation}\label{eq11}
\squared{2.1}\lesssim \norm{f}_{F^s_{p,q}(\Omega)}.
\end{equation}

The second term in \rf{eq1SecondApproach} is the most delicate one. Given cubes $Q$, $S$ and $P$ and points $y\in S$ and $z\in P$ with $1-\varphi_{QS}(z)\neq 0$, we have that $|z-y|\approx \Dist (S,P)$. Therefore, we can write
\begin{align*}
\squared{2.2}
	& = \sum_Q \int_Q \sum_{S\subset\Sh(Q)\setminus 2Q}\int_S \int_{\Sh^4(Q)}\frac{\left|f(z)-f_Q\right||1-\varphi_{QS}(z)|\,dz}{|x-y|^{s+\frac{d}{q}}|y-z|^d}g(x,y) \, dy\, dx\\
\nonumber	& \lesssim \sum_Q \int_Q \sum_{S\in \SH(Q)} \int_{S} \sum_{P\in \SH^4(Q)} \int_{P}\frac{\left|f(z)-f_Q\right|\,dz}{\ell(Q)^{s+\frac{d}{q}}\Dist(S,P)^d}g(x,y) \, dy\, dx .
\end{align*}

Next, we change the focus on the sum. Consider an admissible chain connecting two given cubes $S$ and $P$ both in $\SH^4(Q)$. Then $\Dist(S,P)\approx \ell(S_P)$. Of course, using \rf{eqAdmissible2} and the fact that $S$ and $P$ are in $\SH^4(Q)$ we get 
$$\Dist(Q, S_P)\lesssim \Dist(Q,S)+\Dist(S,S_P)\approx \Dist(Q,S)+\Dist(S,P)\lesssim2\Dist(Q,S)+\Dist(Q,P)\lesssim \ell(Q)$$
  and, therefore, the cube $S_P$ is contained in some $\SH_{\rho_5}(Q)$ for a certain constant $\rho_5$ depending on $d$ and $\varepsilon$. For short, we write $L:=S_P\in\SH^5(Q)$ and $\Sh^5(Q):=\Sh_{\rho_5}(Q)$. Then
\begin{align}\label{eq22InNiceWay}
\squared{2.2}
\nonumber	& \lesssim\sum_Q \int_Q \sum_{L\in\Sh^5(Q)} \sum_{S\in \SH(L)} \int_{S} \sum_{P\in \SH(L)} \int_{P}\frac{\left|f(z)-f_Q\right|\,dz}{\ell(Q)^{s+\frac{d}{q}}\ell(L)^d}g(x,y) \, dy\, dx\\
	& = \sum_Q  \frac{1}{\ell(Q)^{s+\frac{d}{q}}} \int_Q \sum_{L\in\SH^5(Q)} \int_{\Sh(L)} \left|f(z)-f_Q\right|\,dz   \frac{1}{\ell(L)^d} \int_{\Sh(L)} g(x,y) \, dy\, dx.
\end{align}

If we write $g_x(y)=g(x,y)$, we have that for any cube $L$ the integral 
$$\int_{\Sh(L)} g(x,y) \, dy\leq \ell(L)^d \inf_{L} M g_x.$$
Arguing as before, for $\rho_6$ big enough we have that if $L\in SH^4(Q)$, then $\Sh(L)\subset\Sh_{\rho_6}(Q)=:\Sh^6(Q)$ and therefore
$$\int_{\Sh(L)}\left|f(z)-f_Q\right| \, dz =\int_{\Sh(L)}\left|f(z)-f_Q\right| \chi_{\Sh^6(Q)}(z)\, dz \lesssim \int_L M[(f-f_Q)\chi_{\Sh^6(Q)}](\xi)\,d\xi.$$
Back to \rf{eq22InNiceWay} we have that
\begin{align*}
\squared{2.2}
	& \lesssim\sum_Q  \frac{1}{\ell(Q)^{s+\frac{d}{q}}} \int_Q \sum_{L\in\SH^5(Q)}  \int_{L}M[(f-f_Q)\chi_{\Sh^6(Q)}](\xi)  Mg_x(\xi)\,d\xi \, dx\\
	& =\sum_Q  \frac{1}{\ell(Q)^{s+\frac{d}{q}}}   \int_Q\int_{\Sh^5(Q)} M[(f-f_Q)\chi_{\Sh^6(Q)}](\xi)  Mg_x(\xi)\,d\xi \, dx
\end{align*}
and, by H\"older's inequality and the boundedness of the maximal operator in $L^q$ and $L^{q'}$, we have that
\begin{align*}
\squared{2.2}
	& \lesssim \sum_Q \frac{1}{\ell(Q)^{s+\frac{d}{q}}}  \int_Q \left(\int_{\Sh^5(Q)} M[(f-f_Q)\chi_{\Sh^6(Q)}](\xi)^q \,d\xi\right)^\frac{1}{q}\left(\int_{\Sh^5(Q)} Mg_x(\xi)^{q'}\,d\xi\right)^\frac{1}{q'} \, dx\\
	& \lesssim_q \sum_Q \frac{1}{\ell(Q)^{s+\frac{d}{q}}}   \int_Q\left(\int_{\Sh^6(Q)} \left|f(\xi)-f_Q\right|^q \,d\xi\right)^\frac{1}{q}\left(\int_{\Omega} g(x,\xi)^{q'}\,d\xi\right)^\frac{1}{q'} \, dx.
\end{align*}

Again, we write $G(x)=\norm{g(x,\cdot)}_{L^{q'}}$ and by Minkowski's integral inequality we get that
\begin{align*}
\squared{2.2}
	& \lesssim \sum_Q \frac{1}{\ell(Q)^{s+\frac{d}{q}+d}}  \left(\int_{\Sh^6(Q)} \left(\int_Q \left|f(\xi)-f(\zeta) \right|\, d\zeta\right)^q \,d\xi\right)^\frac{1}{q} \int_Q G(x) \, dx\\
	& \lesssim \sum_Q \frac{1}{\ell(Q)^{s+\frac{d}{q}}} \int_Q \left(\int_{\Sh^6(Q)}  \left|f(\xi)-f(\zeta) \right|^q\, d\xi\right)^\frac{1}{q} MG(\zeta) \,d\zeta  .
\end{align*}
Thus, 
\begin{align}\label{eq122}
\squared{2.2}
	& \lesssim \left(\int_\Omega \left(\int_{\Omega} \frac{ \left|f(\xi)-f(\zeta) \right|^q}{|\xi-\zeta|^{sq+d}}\, d\xi\right)^\frac{p}{q}d\zeta\right)^\frac{1}{p}\norm{MG}_{L^{p'}}\lesssim \norm{f}_{F^s_{p,q}(\Omega)}. 
\end{align}

Back to \rf{eq1FirstApproach}, it remains to bound $\squared{B}$ and $\squared{C}$. For the first one, 
$$\squared{B}= \sum_Q \int_Q \sum_{S\subset\Sh(Q)\setminus 2Q} \int_S \frac{\int_{\Omega\setminus\Sh^4(Q)}\left|\lambda_{QQ}(x,z)-\lambda_{QS}(y,z)\right|\,dz}{|x-y|^{s+\frac{d}{q}}}g(x,y) \, dy\, dx,$$ 
just note that if $x\in Q$, $y\in S\in \SH(Q)$ and $z\notin \Sh^4(Q)$ we have that $\varphi_{QQ}(z)=\varphi_{QS}(z)=0$ and, if $\rho_4$ is big enough, $|x-z|>2|x-y|$. Thus, we can use the smoothness condition \rf{eqCalderonZygmundHolder}, that is, $|\lambda_{QQ}(x,z)-\lambda_{QS}(y,z)| \leq \left|K(x-z)-K(y-z)\right|\left|f(z)-f_Q\right|\leq C_K \frac{|f(z)-f_Q||x-y|^s}{|x-z|^{d+s}}$. 

In the last  term in \rf{eq1FirstApproach},
$$\squared{C}=\sum_Q \int_Q \int_{5Q} \frac{\int_{\Omega}\left|\lambda_{QQ}(x,z)-\lambda_{QQ}(y,z)\right|\,dz}{|x-y|^{s+\frac{d}{q}}}g(x,y) \, dy\, dx,$$
 we are integrating in the region where $x\in Q$, $y\in 5Q$ and $z\notin 6Q$ because otherwise $1-\varphi_Q(z)$ would vanish. Also $|x-z|>C_d|x-y|$ and $|x-z|\approx |y-z|$. Thus, we have again that $|\lambda_{QQ}(x,z)-\lambda_{QQ}(y,z)| \leq \left|K(x-z)-K(y-z)\right|\left|f(z)-f_Q\right|\lesssim C_{K} \frac{|f(z)-f_Q||x-y|^s}{|x-z|^{d+s}}$ by \rf{eqCalderonZygmundHolder} and \rf{eqCalderonZygmundSize} (one may use the last one when $2|x-y|\geq |x-z|>C_d|x-y|$, that is $|x-y|\approx |x-z|\approx |y-z|$).

Summing up, 
\begin{align}\label{eq121314}
\squared{B}+\squared{C}
			& \lesssim_{C_K}  \sum_Q \int_Q  \int_{\Sh(Q)} \int_{\Omega\setminus 6Q}\frac{|f(z)-f_Q||x-y|^s\,dz}{|x-y|^{s+\frac{d}{q}}|x-z|^{d+s}}g(x,y) \, dy\, dx=:\squared{2.3}.
\end{align}
with constants depending linearly on the Calder\'on-Zygmund constant $C_{K}$. Reordering, 
$$\squared{2.3} = \sum_Q \int_Q   \int_{\Omega\setminus 6Q}\frac{|f(z)-f_Q|\,dz}{|x-z|^{d+s}}\int_{\Sh(Q)}\frac{g(x,y) \, dy}{|x-y|^{\frac{d}{q}}}\, dx.$$
The last integral above is easy to bound by the same techniques as before: Given $x\in Q \in \mathcal{W}$, since $\frac{d}{q}<d$, by \rf{eqMaximalClose}, H\"older's Inequality and the boundedness of the maximal operator in $L^{q'}$ we have that
\begin{align*}
\int_{\Sh(Q)}\frac{g(x,y) \, dy}{|x-y|^{\frac{d}{q}}}
	& \lesssim \ell(Q)^{d-\frac{d}{q}} \inf_Q Mg_x \leq \ell(Q)^{-\frac{d}{q}}\int_QMg_x
	 \leq \norm{Mg_x}_{L^{q'}} \lesssim_q G(x).
\end{align*}
Thus, 
\begin{align*} 
\squared{2.3}
\nonumber	& \lesssim \sum_Q \int_Q  \sum_P \int_{P}\frac{|f(z)-f_Q|\,dz}{\Dist(P,Q)^{d+s}} G(x)\, dx.
\end{align*}
For every pair of cubes $P,Q\in\mathcal{W}$, there exists an admissible chain $[P,Q]$ and, writing $[P,P_Q)$ for the subchain $[P,P_Q]\setminus\{P_Q\}$ and $[P_Q,Q)$ for  $[P_Q,Q]\setminus\{Q\}$, we get
\begin{align}\label{eq13AdmissibleChains}
\squared{2.3} 
			& \lesssim \sum_Q \int_Q  \sum_P \int_{P}\frac{|f(z)-f_P|\,dz}{\Dist(P,Q)^{d+s}} G(x)\, dx\\
\nonumber			&\quad  + \sum_Q \int_Q  \sum_P \sum_{L\in[P,P_Q)} \frac{|f_L-f_{\mathcal{N}(L)}| \ell(P)^d}{\Dist(P,Q)^{d+s}} G(x)\, dx\\
\nonumber	& \quad + \sum_Q \int_Q  \sum_P \sum_{L\in[P_Q,Q)}\frac{|f_L-f_{\mathcal{N}(L)}|\ell(P)^d}{\Dist(P,Q)^{d+s}} G(x)\, dx=:\squared{2.3.1}+\squared{2.3.2}+\squared{2.3.3}.
\end{align}

The first term in \rf{eq13AdmissibleChains} can be bounded by reordering and using \rf{eqMaximalFar}. Indeed, we have that 
\begin{align*}
\squared{2.3.1} 
\nonumber	& \leq   \sum_P \int_{P}\int_P \frac{|f(z)-f(\xi)|\,d\xi \,dz}{\ell(P)^d} \sum_Q \int_Q \frac{G(x)\, dx}{\Dist(P,Q)^{d+s}} \lesssim   \sum_P \int_{P}\int_P \frac{|f(z)-f(\xi)|\,d\xi MG(z) \,dz}{\ell(P)^{d+s}} ,
\end{align*}
and, by \rf{eqBoundDifferenceAndMaximal} we have that
\begin{align}\label{eq132}
\squared{2.3.1} 
	& \lesssim \norm{f}_{F^s_{p,q}(\Omega)}.
\end{align}

For the second term in \rf{eq13AdmissibleChains} note that given cubes $L\in[P,P_Q]$ we have that $\Dist(P,Q)\approx\Dist(L,Q)$ by \rf{eqNotVeryFar} and $P\in\Sh(L)$ by Definition \ref{defShadow}. Therefore, by  \rf{eqMaximalFar} we have that
\begin{align*}
\squared{2.3.2}
 	& \lesssim  \sum_L \frac{1}{\ell(L)^{2d}}\int_L\int_{5L} |f(\xi)-f(\zeta)| \, d\zeta\,d\xi \sum_{Q} \frac{1}{\Dist(L,Q)^{d+s}}  \int_Q G(x)\, dx  \sum_{P\in\SH(L)}\ell(P)^d \\
	& \lesssim \sum_L \frac{1}{\ell(L)^{2d}}\int_L\int_{5L} |f(\xi)-f(\zeta)| \frac{MG(\zeta)}{\ell(L)^s} \, d\zeta\,d\xi  \ell(L)^d =  \sum_L \int_L\int_{5L} \frac{|f(\xi)-f(\zeta)| MG(\zeta)}{\ell(L)^{d+s}} \, d\zeta\,d\xi  ,
\end{align*}
and, again by \rf{eqBoundDifferenceAndMaximal}, we have that
\begin{align}\label{eq133}
\squared{2.3.2} 
	& \lesssim \norm{f}_{F^s_{p,q}(\Omega)}.
\end{align}

Finally, the last term of \rf{eq13AdmissibleChains} can be bounded analogously: Given cubes $L\in[P_Q,Q]$ we have that $\Dist(Q,P)\approx\Dist(L,P)$ by \rf{eqNotVeryFar}, and
\begin{align*}
\squared{2.3.3}
	& \lesssim  \sum_L \frac{1}{\ell(L)^{2d}}\int_L\int_{5L} |f(\xi)-f(\zeta)| \, d\zeta\,d\xi \sum_{Q\in\SH(L)}   \int_Q G(x)\, dx  \sum_{P} \frac{\ell(P)^d}{\Dist(P,L)^{d+s}} \\
	& \lesssim \sum_L \int_L\int_{5L} |f(\xi)-f(\zeta)| MG(\zeta)  \, d\zeta\,d\xi \frac{ \ell(L)^{d-s}}{\ell(L)^{2d}} =  \sum_L \int_L\int_{5L} \frac{|f(\xi)-f(\zeta)| MG(\zeta)}{\ell(L)^{d+s}} \, d\zeta\,d\xi  ,
\end{align*}
and
\begin{align}\label{eq134}
\squared{2.3.3} 
	& \lesssim \norm{f}_{F^s_{p,q}(\Omega)}.
\end{align}

Now, putting together \rf{eq1FirstApproach},  \rf{eq1SecondApproach}, \rf{eq121314} and \rf{eq13AdmissibleChains} we have that
$$\squared{2}\lesssim_{C_K} \squared{2.1}+\squared{2.2}+\squared{2.3.1}+\squared{2.3.2}+\squared{2.3.3},$$
and by \rf{eq11}, \rf{eq122}, \rf{eq132}, \rf{eq133} and \rf{eq134} we have that 
\begin{equation}\label{eq1done}
\squared{2} \lesssim C_K \norm{f}_{F^s_{p,q}(\Omega)},
\end{equation}
with constants depending on $\varepsilon$, $s$, $p$, $q$ and $d$. Estimates \rf{eq2done} and \rf{eq1done} prove \rf{eqObjective1}.
\end{proof}

%

\begin{proof}[Proof of Theorem \ref{theoT1}]
Let $\Omega$ be a bounded $\varepsilon$-uniform domain. Note that since $s>\frac{d}{p}>\frac{d}{p}-\frac{d}{q}$, we can use the Key Lemma \ref{lemKLemmaT1}, that is,  we have that $T_\Omega$ is bounded if and only if for every $f\in F^s_{p,q}(\Omega)$ we have that 
\begin{equation}\label{eqTargetInT1}
\sum_{Q\in\mathcal{W}} |f_Q|^p \norm{\nabla^s_q T \chi_\Omega}_{L^p(Q)}^p\leq C\norm{f}^p_{F^s_{p,q}(\Omega)},
\end{equation}
with $C$ independent from $f$. Since  $sp>d$, by Lemma \ref{lemExtensionOperator} and Proposition \ref{propoPropertiesBesovTriebel} combined with the Sobolev Embedding Theorem, we have the continuous embedding $F^s_{p,q}(\Omega)\subset L^\infty$. Therefore, given a cube $Q$ we have that $|f_Q|\leq \norm{f}_{L^\infty(\Omega)}\lesssim \norm{f}_{F^s_{p,q}(\Omega)}$ and \rf{eqTargetInT1} holds as long as $T \chi_\Omega\in F^s_{p,q}(\Omega)$.
\end{proof}

To end, we make some observations.
\begin{remark}
In the Key Lemma we have seen that 
\begin{equation}
\sum_{Q\in\mathcal{W}} \norm{ \nabla^s_q T_\Omega (f-f_Q)}_{L^p(Q)}^p \lesssim \left(C_K +  \norm{T}_{F^s_{p,q}\to F^s_{p,q}}+  \norm{T}_{L^p\to L^p}+\norm{T}_{L^q\to L^q}\right)^p \norm{f}^p_{\cdot A^s_{p,q}(\Omega)}.
\end{equation}
Thus, for unbounded domains, we have a $T1$ theorem as well: Let $\Omega\subset \R^d$ be a uniform domain, $T$ a convolution Calder\'on-Zygmund operator of order $0<s< 1$. Consider indices $p,q\in(1,\infty)$ and $\frac{d}{p}<s $. Then 
the truncated operator $T_\Omega$ is bounded in $F^s_{p,q}(\Omega)$ if and only if 
$$\nabla^s_q T_\Omega 1 \in F^s_{p,q}(\Omega)$$
in the sense of \rf{eqInfiniteSupport}.
\end{remark}

\begin{remark}
The Key Lemma is valid in a wider range of indices than Theorem \ref{theoT1} because in the second case we have the restriction of the Sobolev embedding. In the cases where the Key Lemma can be applied but not the theorem above, that is, when $$\max\left\{0,\frac{d}{p}-\frac{d}{q}\right\}<s\leq \min\left\{\sigma,\frac{d}{p}\right\},$$ 
there is room to do some steps forward. 

In \cite[Theorems 1.2 and 1.3]{PratsTolsa}, the authors consider the measures  $\mu_P(x)=|\nabla^s T_\Omega P(x)|^p \, dx$ for polinomials $P$ of degree smaller than the smoothness $s\in \N$ (here the $s$-th gradient has its usual meaning). They conclude that if $\mu_P$ is a $p$-Carleson measure for every such $P$, that is, if 
\begin{equation*}
\int_{\widetilde {\mathbf{Sh}}(a)} \dist(x, \partial\Omega)^{(d-p)(1-p')}(\mu_P({\mathbf{Sh}}(x)\cap {\mathbf{Sh}}(a)))^{p'} \frac{dx}{\dist(x, \partial\Omega)^d}\leq C \mu_P({\mathbf{Sh}}(a)),
\end{equation*}
then $T_\Omega$ is bounded in $W^{s,p}(\Omega)$, and, in case $s=1$, the condition is necessary and sufficient.

Some similar result can be found in the case $\max\left\{0,\frac{d}{p}-\frac{d}{q}\right\}<s\leq \min\left\{\sigma,\frac{d}{p}\right\}$, but is out of the scope of the present article. 

Furthermore, the restriction $\frac{d}{p}-\frac{d}{q}<s$ comes from the intrinsic characterization that we use for the present article, which we think is the easier to handle in our proofs. However, there are other characterizations (see \cite{Strichartz} or \cite[Section 1.11.9]{TriebelTheoryIII}) which cover all the range of indices. There is hope that this characterizations may be used to obtain a result analogous to the Key Lemma \ref{lemKLemmaT1} for a wider range.
\end{remark}

\begin{remark}
For $1<p,q<\infty$ and $0<s<\frac{1}{p}$, we have that the multiplication by the characteristic functions of a half plane is bounded in $F^s_{p,q}(\R^d)$. This implies that for domains $\Omega$ whose boundary consists on a finite number of polygonal boundaries, the pointwise multiplication with $\chi_\Omega$ is also bounded and, using characterizations by differences, this property can be seen to be stable under bi-Lipschitz changes of coordinates. Summing up, given any Lipschitz domain $\Omega$ and any function $f\in F^s_{p,q}(\R^d)$, we have that
$$\norm{\chi_\Omega \, f}_{F^s_{p,q}(\R^d)}\lesssim \norm{f}_{F^s_{p,q}(\R^d)}.$$
Therefore, if $s>\frac{d}{p}-\frac{d}{q}$ and $T$ is an  operator bounded  in $F^s_{p,q}$, using the extension $\Lambda_0:F^s_{p,q}(\Omega)\to F^s_{p,q}(\R^d)$ (see Corollary \ref{coroExtensionDomain}), for every $f\in F^s_{p,q}(\Omega)$ we have that 
\begin{align*}
\norm{T_\Omega f}_{F^s_{p,q}(\Omega)}
	& = \norm{T(\chi_\Omega\, \Lambda_0 f)}_{F^s_{p,q}(\Omega)} \leq \norm{T(\chi_\Omega\, \Lambda_0 f)}_{F^s_{p,q}}\leq \norm{T}_{F^s_{p,q}\to F^s_{p,q}} \norm{\chi_\Omega\, \Lambda_0 f}_{F^s_{p,q}} \lesssim \norm{\Lambda_0 f}_{F^s_{p,q}}\\
	& \lesssim \norm{f}_{F^s_{p,q}(\Omega)}.
\end{align*}
In particular, given a convolution Calder\'on-Zygmund operator $T$ and a Lipschitz domain $\Omega$ we have that $T_\Omega$ is bounded in $F^s_{p,q}(\Omega)$ for any $0<s<\frac{1}{p}$.
\end{remark}

\renewcommand{\abstractname}{Acknowledgements}
\begin{abstract}
The authors want to express their gratitude towards Hans Triebel who gave them some hints on how to extend their results on $W^{s,p}$ and $B^s_{p,p}$ to the wider class of spaces $F^s_{p,q}(\Omega)$ when $s>\max\left\{0,\frac{d}{p}-\frac{d}{q}\right\}$. We are also grateful to Antti V\"ah\"akangas who brought the reference \cite[Proposition 5]{Dyda} to our knowledge and hence gave the inspiration for Lemma \ref{lemSecondReduction}, and Jarod Hart for his advice on Fourier multipliers regarding Triebel-Lizorkin spaces.

They also want to thank Irene Drelichman for bringing to our attention the result \cite[Corollary 1]{Seeger}, see Remark \ref{remSeeger}.

The first author was funded by the European Research
Council under the European Union's Seventh Framework Programme (FP7/2007-2013) /
ERC Grant agreement 320501. Also,  partially supported by grants 2014-SGR-75 (Generalitat de Catalunya), MTM-2010-16232 and MTM-2013-44304-P (Spanish government) and by a FI-DGR grant from the Generalitat de Catalunya, (2014FI-B2 00107).

The second author was supported by the Finnish Academy via the The Finnish Centre of Excellence (CoE) in Analysis and Dynamics Research.\end{abstract}

\bibliography{../../../bibtex/llibres}
\end{document}